\documentclass[12pt]{amsart}
\usepackage{amscd,amssymb,amsthm,verbatim}
\newcommand{\dvol}{\operatorname{dvol}}

\newcommand{\Hess}{\operatorname{Hess}}

\newcommand{\N}{{\mathbb N}}

\newcommand{\R}{{\mathbb R}}

\newcommand{\Ric}{\operatorname{Ric}}

\newcommand{\Tr}{\operatorname{Tr}}

\newcommand{\vol}{\operatorname{vol}}
\newcommand{\Z}{{\mathbb Z}}

\numberwithin{equation}{section}
\theoremstyle{plain}
\newtheorem{definition}{Definition}

\newtheorem{lemma}{Lemma}
\newtheorem{theorem}{Theorem}
\newtheorem{proposition}{Proposition}
\newtheorem{corollary}{Corollary}

\theoremstyle{remark}

\newtheorem{remark}{Remark}
\newtheorem{example}{Example}

\usepackage{graphicx}
\input{amssym.def}
\input{amssym}
\input{epsf}
\usepackage{a4wide}

\def\R{\mathbb R}
\def\N{\mathbb N}

\def\Z{\mathbb Z}

\def\Hess{\mathop{{\rm Hess}\,}}

\def\2dr#1#2{\left. \frac{d^2}{d{#1}^2} \right |_{#2}}
\def\d2#1{\frac{d^2}{d{#1}^2}}

\DeclareMathOperator*{\grad}{grad}

\def\begeq{\begin{equation} \label{} \label{}}
\def\endeq{\end{equation}}
\def\begar{\begin{eqnarray}}
\def\endar{\end{eqnarray}}
\def\begar*{\begin{eqnarray*}}
\def\endar*{\end{eqnarray*}}
\def\begal{\begin{align} \label{} \label{}}
\def\endal{\end{align}}
\def\begal*{\begin{align*}}
\def\endal*{\end{align*}}

\theoremstyle{definition}

\theoremstyle{remark}

\newtheorem*{Thm*}{Theorem}
\newtheorem*{Lem*}{Lemma}
\newtheorem*{Conj*}{Conjecture}
\newtheorem*{Cor*}{Corollary}
\newtheorem*{Def*}{Definition}
\newtheorem*{Prop*}{Proposition}
\newtheorem*{Exo*}{Exercise}
\newtheorem*{Exs*}{Examples}
\newtheorem*{Ex*}{Example}
\newtheorem*{Rk*}{Remark}
\newtheorem*{Rks*}{Remarks}

\begin{document}

\title[Optimal transport and Perelman's reduced volume]
{Optimal transport and Perelman's reduced volume}

\author{John Lott}
\address{Department of Mathematics\\
University of Michigan\\
Ann Arbor, MI  48109-1109\\
USA} \email{lott@umich.edu}

\thanks{This research was partially 
supported by NSF grant DMS-0604829}
\date{December 8, 2008}

\begin{abstract}
We show that a certain entropy-like function is convex, under an
optimal transport problem that is adapted to Ricci flow.
We use this to reprove the monotonicity of Perelman's reduced
volume.
\end{abstract}

\maketitle
\section{Introduction}

One of the major tools introduced by Perelman is his
reduced volume $\widetilde{V}$ \cite[Section 7]{Perelman}.  
This is a certain geometric quantity which is
monotonically nondecreasing in time when one has a Ricci flow
solution.  Perelman's main use of the reduced
volume was to rule out local collapsing in a Ricci flow.

Before giving his rigorous proof that $\widetilde{V}$ is monotonic,
Perelman gave a heuristic argument \cite[Section 6]{Perelman}.
Given a Ricci flow
solution $(M, g(\tau))$ on a compact manifold $M$, 
where $\tau$ is backward time, Perelman
considered the manifold $\widetilde{M} = M \times S^N \times \R^+$
with the Riemannian metric
\begin{equation} \label{1.1}
\widetilde{g} = g(\tau) \: + \: 2N \tau g_{S^N} \: + \:
\left( \frac{N}{2\tau} \: + \: R \right) \: d\tau^2.
\end{equation}
Here $R$ denotes the scalar curvature and $g_{S^N}$ is the metric on $S^N$ with
constant sectional curvature $1$.
Perelman showed that the Ricci curvatures of $\widetilde{M}$ vanish to
leading order in $N$. Now the Bishop-Gromov inequality says that
if a complete Riemannian  manifold $Z$ has nonnegative Ricci curvature then 
$r^{- \dim(Z)} \: \vol(B_r(z))$ is nonincreasing in $r$. 
Perelman formally
applied the Bishop-Gromov inequality to $\widetilde{M}$, 
translated the result back
down to $M$ and took the
limit when $N \rightarrow \infty$, to get the monotonicity of $\widetilde{V}$.

In another direction, there has been recent work showing the
equivalence between the nonnegative Ricci curvature of a Riemannian
manifold $M$, and the convexity (in time) of certain
entropy functions in an optimal transport problem on $M$
\cite{Cordero-Erausquin-McCann-Schmuckenschlaeger 
(2001),Lott-Villani,Otto-Villani (2000),Sturm,Sturm2,Sturm-von Renesse 
(2005)}. A survey is in \cite{Lott (2008b)} and a
detailed exposition is in Villani's book \cite{Villani2}.
(Background information on optimal transport is in Villani's books
\cite{Villani,Villani2}.)
In view of Perelman's heuristic argument, it is natural to wonder
whether having a Ricci flow solution $(M, g(t))$ implies the convexity
of an entropy in some optimal transport problem on $M$.
The idea is that the asymptotic nonnegative Ricci curvature on
$\widetilde{M}$ should imply the asymptotic convexity of the entropy in an
optimal transport problem on $\widetilde{M}$, which should then
translate to a statement about optimal transport on $M$.

It turns out that this can be done. 
The optimal transport problem on $M$ has a cost function coming from
Perelman's ${\mathcal L}$-functional. This sort of transport
problem was introduced by Topping \cite{Topping}, as described below,
with the purpose of constructing certain monotonic quantities
for a Ricci flow.
Bernard-Buffoni
\cite{Bernard-Buffoni (2007)} and 
Villani \cite[Chapters 7,10,13]{Villani2} gave
analytic results for
general time-dependent cost functions.

In fact, there are three relevant costs for Ricci flow : one corresponding to
Perelman's ${\mathcal L}$-functional (which we will call ${\mathcal L}_-$), 
one corresponding to the Feldman-Ilmanen-Ni ${\mathcal L_+}$-functional
\cite{Feldman-Ilmanen-Ni (2005)} and a third one which we call
${\mathcal L}_0$. 
In the case of the ${\mathcal L}_-$-cost, the main result of the
paper is the following.

\begin{theorem} \label{thm1}
Suppose that $(M, g(\tau))$ is a Ricci flow solution on a 
connected closed $n$-dimensional manifold $M$, 
where $\tau$ denotes
backward time. Let $c(\tau)$ be the displacement interpolation in an
optimal transport problem between absolutely continuous
probability measures
$c(\tau_0)$ and $c(\tau_1)$, with ${\mathcal L}_-$-cost. Then 
${\mathcal E}(c(\tau)) \: + \: \int_M \phi(\tau) \: dc(\tau) \: + \:
\frac{n}{2} \: \log(\tau)$ is convex in the variable $s = \tau^{- \: \frac12}$.
\end{theorem}

Here ${\mathcal E}(c(\tau))$ is the (negative) relative entropy
of $c(\tau)$ with respect to the time-$\tau$ Riemannian volume density.
The function $\phi(\tau)$ is the potential for the velocity field
in the displacement
interpolation.

We show that the monotonicity of Perelman's reduced volume
$\widetilde{V}$ is a consequence of Theorem \ref{thm1}; see
Corollary \ref{cor8}.

There are two main approaches to optimal transport problems :
the Eulerian approach and the Lagrangian approach. Let
$P(M)$ denote the Borel probability measures on a static
Riemannian manifold $M$ and let
$P^\infty(M)$ denote those with a smooth positive density.
The Eulerian approach of Benamou-Brenier considers smooth maps
$c \: : \: [t_0, t_1] \rightarrow P^\infty(M)$ that
minimize an action $E(c)$, among all such curves
with the same endpoints \cite{Benamou-Brenier (2000)}.
In the associated Otto calculus, one considers $P^\infty(M)$ to be an
infinite-dimensional Riemannian manifold and $E(c)$
to be the corresponding energy of the curve $c$, so the
Euler-Lagrange equation for $E$ becomes the
geodesic equation on $P^\infty(M)$ \cite{Otto (2001)}.
Otto and Villani used this approach to compute the
time-derivatives of the entropy function ${\mathcal E}$ along
the curve $c$ \cite{Otto-Villani (2000)}. 

The Lagrangian approach to optimal transport considers
a displacement interpolation $c$, i.e. a geodesic
in the Eulerian approach, to be specified by the
family of geodesics in $M$ that describe the trajectories
taken by particles in the original mass distribution
$c(t_0)$, when transporting it to the final mass distribution
$c(t_1)$. In the case of optimal transport on Riemannian manifolds,
the Lagrangian approach was developed by McCann \cite{McCann (2001)}
and Cordero-Erausquin, McCann and Schmuckenschl\"ager
\cite{Cordero-Erausquin-McCann-Schmuckenschlaeger 
(2001)}.

Comparing the two approaches, the Eulerian approach is perhaps more
insightful whereas the Lagrangian approach is better suited to
deal with the regularity issues that arise in
optimal transport. (See, however, the
papers of Daneri-Savar\'e \cite{Daneri-Savare (2008)} and
Otto-Westdickenberg \cite{Otto-Westdickenberg (2005)}, which
prove results about optimal transport in $P(M)$ using the
Eulerian approach along with density arguments.) Much of
the present paper consists of describing an Otto calculus
which is adapted for the optimal transport of measures under
a Ricci flow background.

There has been earlier work relating optimal transport to Ricci flow.
The author \cite{Lott (2005)} and McCann-Topping \cite{McCann-Topping}
observed that under a Ricci flow background,
if $c_1(t)$ and $c_2(t)$ are solutions of the
backward heat equation on $P(M)$ then the Wasserstein distance
$W_2(c_1(t), c_2(t))$ is monotonically nondecreasing in $t$.
A detailed proof using the Lagrangian approach
appears in \cite{McCann-Topping}.  McCann-Topping
noted that this monotonicity property characterizes supersolutions
to the Ricci flow equation. In follow-up work, Topping
considered optimal transport with the ${\mathcal L}_-$-cost function
and showed the monotonicity of a certain distance function between
the measures $c_1$ and $c_2$, when taken at different but related times.
We refer to \cite{Topping} for the precise statement.
He then used this to rederive the monotonicity of Perelman's
${\mathcal W}$-functional. In the Lagrangian proof of Theorem \ref{thm1}
we use Topping's calculations for the $\tau$-derivatives
of ${\mathcal E}(c(\tau))$; see Remark \ref{addedremark}.

The outline of this paper is as follows. In Section \ref{section2}
we review the Otto calculus for optimal transport
on a manifold with a time-independent Riemannian metric.
In Section \ref{section3} we use the Otto calculus to 
prove that if $(M, g(t))$ is a Ricci flow solution and
$c_1(t), c_2(t)$ are solutions of the
backward heat equation in $P^\infty(M)$ then the Wasserstein distance
$W_2(c_1(t), c_2(t))$ is monotonically nondecreasing in $t$.
In Section \ref{section4} we introduce the ${\mathcal L}_0$-cost.
We give an Otto calculus for optimal transport with 
${\mathcal L}_0$-cost, under a background Ricci flow solution.
We then show the ${\mathcal L}_0$-analog of Theorem \ref{thm1} above.
In Section \ref{section5} we give the ${\mathcal L}_0$-analog of
Topping's monotonicity statement regarding the distance between
two solutions of the backward heat equation on measures.
We use this to reprove the monotonicity of Perelman's 
${\mathcal F}$-functional.
In Section \ref{section6} we give an Otto calculus for optimal transport with 
${\mathcal L}_-$-cost, under a background Ricci flow solution.
In Section \ref{section7} we prove Theorem \ref{thm1} and we use it 
to reprove the monotonicity of Perelman's reduced volume.
In Section \ref{section8} we discuss what Ricci flow should mean on
a smooth metric-measure space. In Appendix \ref{section9} we indicate
how the results of Sections \ref{section6} and \ref{section7} extend to the
${\mathcal L}_+$-cost.

Regarding the overall method of proof in this paper,
calculations in the Eulerian formalism
can be considered to be either rigorous statements on $P^\infty(M)$ or formal
statements on $P(M)$. When a suitable density result is available, one
can use the Eulerian methods to give rigorous proofs on $P(M)$.
In this way, we give rigorous Eulerian proofs
on $P(M)$ of the statements in Sections \ref{section2} and \ref{section3}, making use of the
nontrivial Otto-Westdickenberg density
result \cite{Otto-Westdickenberg (2005)}. 
Sections \ref{section4}-\ref{section7} contain calculations in the Eulerian
framework under a Ricci flow background.
We expect that one can extend these calculations to rigorous proofs on $P(M)$, by
adapting the density methods of \cite{Daneri-Savare (2008)} 
or \cite{Otto-Westdickenberg (2005)} to the setting of time-dependent cost
functions. 
We do not address this issue here.  Consequently, we revert to Lagrangian
methods when we want to give rigorous proofs
in $P(M)$ of the statements in Sections \ref{section4}-\ref{section7}.

I thank Peter Topping and C\'edric Villani for discussions, and the referee for
helpful comments.
I thank the UC-Berkeley Mathematics Department and the IHES for 
their hospitality while part of this
research was performed.

\section{Otto calculus} \label{section2}

This section is mostly concerned with known results about
optimal transport on a fixed Riemannian manifold $M$. It is a warmup for
the later sections, which extend
the results to the case when the Riemannian metric evolves under the 
Ricci flow.

We use the
Otto calculus to give rigorous proofs of certain statements about the space of
smooth probability measures $P^\infty(M)$. These proofs can
then be considered as formal proofs of the analogous statements
on the space of all probability measures $P(M)$. The rigorous
proofs of the statements on $P(M)$ are usually done by the Lagrangian
approach, but one can also use the density of $P^\infty(M)$ in
$P(M)$
\cite{Daneri-Savare (2008),Otto-Westdickenberg (2005)}.
Most of the calculations in this section can be extracted from
\cite{Otto-Villani (2000)} and \cite{Otto-Westdickenberg (2005)}.

In what follows, we use the Einstein summation convention freely.

Let $(M,g)$ be a smooth connected closed
(= compact boundaryless) Riemannian manifold of dimension $n > 0$. 
We denote the Riemannian
density by $\dvol_M$. Let $P(M)$ denote the space of Borel probability
measures on $M$, equipped with the Wasserstein metric $W_2$. For relevant
results about optimal transport and 
the Wasserstein metric, we refer to \cite[Sections 1 and 2]{Lott-Villani}
and references therein.
A fuller exposition is in the books
\cite{Villani} and \cite{Villani2}.
As $P(M)$ is a length space, it makes sense
to talk about geodesics in $P(M)$, which we will always take to be
minimizing and parametrized proportionately to arc-length.

Put 
\begin{equation} \label{2.1} 
P^\infty(M) \: = \: \{ \rho \: \dvol_M \: : \: \rho \in C^\infty(M), \rho > 0, \int_M \rho \:
\dvol_M \: = \: 1\}.
\end{equation}
Then $P^\infty(M)$ is a dense subset of $P(M)$, 
as is the complement of
$P^\infty(M)$ in $P(M)$.
For the purposes of this paper,
we give $P^\infty(M)$ the smooth topology. (This differs from the subspace topology on
$P^\infty(M)$ coming from its inclusion in $P(M)$.) Then $P^\infty(M)$ has the structure of an
infinite-dimensional smooth manifold in the sense of
\cite{Kriegl-Michor}. The formal calculations in this section are rigorous
calculations on the smooth manifold $P^\infty(M)$.

Given $\phi \in C^\infty(M)$, define a vector field $V_\phi$ on
$P^\infty(M)$ by saying that for $F \in C^\infty(P^\infty(M))$,
\begin{align} \label{2.2}
(V_\phi F)(\rho \dvol_M) \: & = \: 
\frac{d}{d\epsilon} \Big|_{\epsilon = 0} 
F \left( \rho \dvol_M \: - \: \epsilon \: \nabla^i ( \rho \nabla_i \phi) \dvol_M \right) \\
&= \: \frac{d}{d\epsilon} \Big|_{\epsilon = 0} 
F \left( \Phi_*^\epsilon (\rho \dvol_M) \right), \notag
\end{align}
where
$\Phi^\epsilon(m) = \exp_m (\epsilon \nabla_m \phi)$.
The map $\phi \rightarrow V_\phi$ passes to an isomorphism
$C^\infty(M)/\R \rightarrow T_{\rho \dvol_M} P^\infty(M)$.
This parametrization of $T_{\rho \dvol_M} P^\infty(M)$ goes back to
Otto's paper
\cite{Otto (2001)}; see \cite{Ambrosio-Gigli-Savare (2004)}
for further discussion.
Otto's Riemannian metric $G$ on $P^\infty(M)$ is given \cite{Otto (2001)} by
\begin{align} \label{2.3}
G( V_{\phi_1}, V_{\phi_2}) (\rho \dvol_M) \: & = \:
\int_M \langle \nabla \phi_1, \nabla \phi_2 \rangle \: \rho \: \dvol_M \\
&  = \: - \: 
\int_M  \phi_1 \nabla^i ( \rho \nabla_i \phi_2) \: \dvol_M. \notag
\end{align}
In view of (\ref{2.2}), we write $\delta_{V_{\phi}} \rho \: = \: - \:  \nabla^i ( \rho \nabla_i \phi)$.
Then
\begin{equation} \label{2.4} 
G (V_{\phi_1}, V_{\phi_2}) (\rho \dvol_M) \:  = \:
\int_M  \phi_1 \: \delta_{V_{\phi_2}} \rho \: \dvol_M
\: = \: \int_M  \phi_2 \: \delta_{V_{\phi_1}} \rho \: \dvol_M.
\end{equation}

We now relate the Riemannian metric $G$ to the Wasserstein metric $W_2$.
In \cite{Otto-Villani (2000)} it was heuristically
shown that the geodesic distance coming from (\ref{2.4})
equals the Wasserstein metric.  To give a rigorous relation,
we recall that a curve 
$c \: : [0,1] \rightarrow P(M)$ has a length given by 
\begin{equation} \label{2.5}
L(c) \: = 
\: \sup_{J\in \N}\; \sup_{0 = s_0 \le s_1 \le \ldots \le s_J = 1}
\sum_{j=1}^{J} W_2\bigl(c(s_{j-1}), c(s_{j})\bigr).
\end{equation}
From the triangle inequality, the expression $\sum_{j=1}^{J} W_2\bigl(c(s_{j-1}), c(s_{j})\bigr)$
is nondecreasing under a refinement of the partition $0 = s_0 \le s_1 \le \ldots \le s_J = 1$.

If $c \: : [0,1] \rightarrow P^\infty(M)$ is a smooth curve in $P^\infty(M)$ 
then 
we write $c(s) \: = \: \rho(s) \: \dvol_M$ and let $\phi(s) \in C^\infty(M)$ 
satisfy
\begin{equation} \label{2.6}
\frac{\partial \rho}{\partial s} \: = \: - \: 
\nabla^i \left( \rho \nabla_i \phi \right).
\end{equation}
It is easy to see, using the spectral theory of the
weighted Laplacian on $L^2(M, \rho(s) \: \dvol_M)$, that
$\phi(s)$ exists.
Note that $\phi(s)$ is uniquely defined up to an additive constant.
The Riemannian length of $c$, as
computed using (\ref{2.3}), is
\begin{equation} \label{2.7}
\int_0^1 \sqrt{G(c^\prime(s), c^\prime(s))} \: ds \: = \:
\int_0^1 \left( \int_M |\nabla \phi(s)|^2 \: \rho(s) \: \dvol_M \right)^{\frac12} \: ds.
\end{equation}
\begin{theorem}  \cite[Proposition 1]{Lott (2008)}
If $c \: : [0,1] \rightarrow P^\infty(M)$ is a smooth 
immersed curve then the two notions of length agree, in the sense that
\begin{equation} \label{2.8} 
L(c) \: = \: \int_0^1 \sqrt{G( c^\prime(s), c^\prime(s) )} \: ds.
\end{equation}
\end{theorem}

Next, consider the Lagrangian 
\begin{equation} \label{2.9}
E(c) \: = \: \frac12 \int_0^1 G( c^\prime(s), c^\prime(s) ) \: ds \: = \:
\frac12 \int_0^1 \int_M |\nabla \phi(s)|^2 \: \rho(s) \: \dvol_M \: ds.
\end{equation}

\begin{theorem} \label{OW} \cite[Proposition 4.3]{Otto-Westdickenberg (2005)}
Fix measures $\rho_0 \: \dvol_M, \rho_1 \: \dvol_M \in P^\infty(M)$. 
Then
the infimum of $E$, over smooth paths in $P^\infty(M)$ with those endpoints, is
$\frac12 \: W_2(\rho_0 \: \dvol_M, \rho_1 \: \dvol_M)^2$.
\end{theorem}

In general we cannot replace the
``inf'' in the statement of Theorem \ref{OW}
by ``min'', since the Wasserstein geodesic connecting
$\rho_0 \: \dvol_M$ and $\rho_1 \: \dvol_M$ may not lie entirely
in $P^\infty(M)$.

We now compute the first variation of $E$.

\begin{proposition} \label{prop1}
Let 
\begin{equation} \label{2.10}
\rho \: \dvol_M \: : \: [0,1] \times [t_0 - \epsilon, t_0 + \epsilon]
\rightarrow P^\infty(M)
\end{equation}
be a smooth map, with $\rho \equiv \rho(s,t)$.
Let 
\begin{equation} \label{2.11}
\phi \: : \: [0,1] \times [t_0 - \epsilon, t_0 + \epsilon]
\rightarrow C^\infty(M)
\end{equation}
be a smooth map that satisfies (\ref{2.6}), with
$\phi \equiv \phi(s,t)$.
Then 
\begin{equation} \label{2.12}
\frac{dE}{dt} \Bigg|_{t = t_0} \: = \:
\int_M \phi \: \frac{\partial \rho}{\partial t}
\: \dvol_M \Bigg|_{s=0}^1 \: - \: \int_0^1
\int_M \left( \frac{\partial \phi}{\partial s} \: + \: \frac12 \:
|\nabla \phi|^2 \right) \: \frac{\partial \rho}{\partial t} \: \dvol_M \: ds,
\end{equation}
where the right-hand side is evaluated at time $t = t_0$.
\end{proposition}
\begin{proof}
We have
\begin{equation} \label{2.13}
\frac{dE}{dt} \: = \: \int_0^1 \int_M \left( \langle \nabla \phi,
\nabla \frac{\partial \phi}{\partial t} \rangle \: \rho \: + \: 
\frac12 \: |\nabla \phi|^2 \: \frac{\partial \rho}{\partial t} \right)
\: \dvol_M \: ds. 
\end{equation}
For a fixed $f \in C^\infty(M)$, from (\ref{2.6}),
\begin{equation} \label{2.14}
\int_M f \: \frac{\partial \rho}{\partial s} \: \dvol_M \: = \:
\int_M \langle \nabla f, \nabla \phi \rangle \: \rho \: \dvol_M.
\end{equation}
Hence
\begin{equation} \label{2.15}
\int_M f \: \frac{\partial^2 \rho}{\partial s \partial t} \: \dvol_M \: = \:
\int_M \left( \langle \nabla f, \nabla \frac{\partial \phi}{\partial t} 
\rangle \: \rho \: + \:
\langle \nabla f, \nabla \phi \rangle \: \frac{\partial \rho}{\partial t}
\right)
\: \dvol_M.
\end{equation}
Taking $f = \phi$ gives
\begin{equation} \label{2.16}
\int_M \phi \: \frac{\partial^2 \rho}{\partial s \partial t} \: \dvol_M \: = \:
\int_M \left( \langle \nabla \phi, \nabla \frac{\partial \phi}{\partial t} 
\rangle \: \rho \: + \:
| \nabla \phi|^2 \: \frac{\partial \rho}{\partial t}
\right)
\: \dvol_M.
\end{equation}
Equations (\ref{2.13}) and (\ref{2.16}) give
\begin{align} \label{2.17}
\frac{dE}{dt} \: & = \: \int_0^1 \int_M \left( 
\phi \: \frac{\partial^2 \rho}{\partial s \partial t} \: - \: \frac12 \:
| \nabla \phi|^2 \: \frac{\partial \rho}{\partial t}
\right)
\: \dvol_M \: ds \\
 & = \: \int_0^1 \int_M \left( \frac{\partial}{\partial s} \left(
\phi \: \frac{\partial \rho}{\partial t} \right) \: - \:
\left(\frac{\partial \phi}{\partial s}  \: + \: \frac12 \:
| \nabla \phi|^2 \right) \: \frac{\partial \rho}{\partial t}
\right)
\: \dvol_M \: ds, \notag
\end{align}
from which the proposition follows.
\end{proof}

From (\ref{2.12}), the Euler-Lagrange equation for $E$ is 
\begin{equation} \label{2.18}
\frac{\partial \phi}{\partial s} \: = \: - \: \frac12 \: |\nabla \phi|^2
\: + \: \alpha(s),
\end{equation}
where $\alpha \in C^\infty([0,1])$. Changing $\phi$ by a spatially-constant
function, we can assume that $\alpha = 0$, so the Euler-Lagrange equation
for $E$ becomes the Hamilton-Jacobi equation
\begin{equation} \label{2.19}
\frac{\partial \phi}{\partial s} \: = \: - \: \frac12 \: |\nabla \phi|^2.
\end{equation}
If a geodesic in $P(M)$ happens to be a smooth curve in
$P^\infty(M)$ then it will satisfy (\ref{2.19}).

For any $0 \le s^\prime < s^{\prime \prime} \le 1$, the viscosity solution of 
(\ref{2.19}) satisfies
\begin{equation} \label{2.20}
\phi(s^{\prime \prime})(m^{\prime \prime}) 
\: = \: \inf_{m^\prime \in M} \left( \phi(s^\prime)(m^\prime) \: + \:
\frac{d_M(m^\prime, m^{\prime \prime})^2}{s^{\prime \prime} - s^\prime} 
\right). 
\end{equation}
Then the solution of (\ref{2.6}) satisfies 
\begin{equation} \label{2.21}
\rho(s^{\prime \prime}) \dvol_M \: = \: 
(F_{s^\prime, s^{\prime \prime}})_* (\rho(s^\prime) \: \dvol_M), 
\end{equation}
where the transport map
$F_{s^\prime, s^{\prime \prime}} \: : \: M \rightarrow M$ is given by
\begin{equation} \label{2.22}
F_{s^\prime, s^{\prime \prime}}(m^\prime) \: = \: \exp_{m^\prime} 
\left( (s^{\prime \prime} - s^\prime) \nabla_{m^\prime}
\phi(s^\prime) \right).
\end{equation}

We now give some simple results in the Otto calculus.

\begin{proposition} \label{prop2}
Assuming (\ref{2.6}) and (\ref{2.19}),
we have
\begin{equation} \label{2.23}
\frac{d}{ds} \int_M \phi \: \rho \: \dvol_M \: 
= \:
\frac12 \: \int_M |\nabla \phi|^2 \: \rho \: \dvol_M
\end{equation}
and
\begin{equation} \label{2.24}
\frac12 \: \frac{d}{ds} 
 \int_M |\nabla \phi|^2 \: \rho \: \dvol_M \: = \: 0.
\end{equation}
\end{proposition}
\begin{proof}
First,
\begin{align} \label{2.25}
\frac{d}{ds} \int_M \phi \: \rho \: \dvol_M \: & = \:
- \: \frac12 \: \int_M |\nabla \phi|^2 \: \rho \: \dvol_M \: - \:
\int_M \phi \: \nabla^i ( \rho \nabla_i \phi) \: \dvol_M \\
& = \:
\frac12 \: \int_M |\nabla \phi|^2 \: \rho \: \dvol_M \notag.
\end{align}
Next, using (\ref{2.18}),
\begin{align} \label{2.26}
& \frac12 \: \frac{d}{ds} 
 \int_M |\nabla \phi|^2 \: \rho \: \dvol_M \: = \\
& - \: \frac12
 \int_M \langle \nabla \phi,  \: \nabla (|\nabla \phi|^2) \rangle \:
\rho \: \dvol_M \: - \: \frac12
 \int_M |\nabla \phi|^2 \: \nabla^i ( \rho \nabla_i \phi) \: \dvol_M 
\: = 0 \notag.
\end{align}
This proves the proposition.
\end{proof}

Equation (\ref{2.24}) is just the statement that a geodesic in $P^\infty(M)$ has
constant speed. Equation (\ref{2.23}) says that
$\int_M \phi \: \rho \: \dvol_M$ is proportionate to the arc length
along the geodesic.

The (negative)
entropy ${\mathcal E} \: : \: P^\infty(M) \rightarrow \R$ is given by
\begin{equation} \label{2.27}
{\mathcal E}(\rho \: \dvol_M) \: = \: \int_M \rho \: \log(\rho) \: \dvol_M.
\end{equation}
We now compute its first two derivatives along a curve in $P^\infty(M)$.

\begin{proposition} \label{prop3}
Assuming (\ref{2.6}), 
we have
\begin{equation} \label{2.28}
\frac{d {\mathcal E}}{ds} \: = \:
\int_M \langle
\nabla \phi, \nabla \rho \rangle \: \dvol_M \: = \:
- \: \int_M 
\nabla^2 \phi \: \rho \: \dvol_M
\end{equation}
and
\begin{align} \label{2.29}
\frac{d^2 {\mathcal E}}{ds^2} \: & = \:
- \: \int_M \ \left( \frac{\partial \phi}{\partial s} \: + \:
\frac12 \: |\nabla \phi|^2 \right) \: \nabla^2 \rho \: \dvol_M \: + \\
& \: \: \: \: \: \: 
\int_M \left(
|\Hess \phi|^2 \: + \: \Ric(\nabla \phi,\nabla \phi) \right)
\: \rho \: \dvol_M. \notag
\end{align}
\end{proposition}
\begin{proof}
First,
\begin{equation} \label{2.30}
\frac{d {\mathcal E}}{ds} \: = \: - \: 
\int_M \left( \log(\rho) + 1 \right) \: \nabla^i (\rho \: \nabla_i \phi) \:
\dvol_M \: = \: 
\int_M \langle
\nabla \phi, \nabla \rho \rangle \: \dvol_M.
\end{equation}
Then
\begin{align} \label{2.31}
\frac{d^2 {\mathcal E}}{ds^2} \: & = \:
\int_M \left\langle \nabla \left( \frac{\partial \phi}{\partial s} \: + \:
\frac12 \: |\nabla \phi|^2 \right), \nabla \rho \right\rangle \: \dvol_M \: - \\
& \: \: \: \: \: \: \frac12 \:
\int_M \left\langle \nabla \left(
|\nabla \phi|^2 \right), \nabla \rho \right\rangle \: \dvol_M
\: + \: \int_M \left\langle \nabla \phi, \nabla \left( - \nabla^i (
\rho \nabla_i \phi) \right) \right\rangle \: \dvol_M \notag \\
& = \:
- \: \int_M \ \left( \frac{\partial \phi}{\partial s} \: + \:
\frac12 \: |\nabla \phi|^2 \right) \: \nabla^2 \rho \: \dvol_M \: + \notag \\
& \: \: \: \: \: \: \frac12 \:
\int_M \nabla^2 \left(
|\nabla \phi|^2 \right) \: \rho \: \dvol_M
\: - \: \int_M \left\langle \nabla \nabla^2 \phi, \nabla  \phi
\right\rangle \: \rho \: \dvol_M \notag \\
& = \:
- \: \int_M \ \left( \frac{\partial \phi}{\partial s} \: + \:
\frac12 \: |\nabla \phi|^2 \right) \: \nabla^2 \rho \: \dvol_M \: + \notag \\
& \: \: \: \: \: \: 
\int_M \left(
|\Hess \phi|^2 \: + \: \Ric(\nabla \phi,\nabla \phi) \right)
\: \rho \: \dvol_M, \notag
\end{align}
where we used the Bochner identity in the last line.
This proves the proposition.
\end{proof}

\begin{corollary} \label{cor1}
Assuming (\ref{2.6}) and (\ref{2.19}), if $\Ric(M, g) \ge 0$ then 
$\frac{d^2 {\mathcal E}}{ds^2} \: \ge \: 0$.
That is, ${\mathcal E}$ is convex along geodesics in $P^\infty(M)$.
\end{corollary}

\begin{remark}
In view of Proposition \ref{prop2}, Corollary \ref{cor1}
would still hold if we replaced
${\mathcal E}(\rho(s) \: \dvol_M)$ by
${\mathcal E}(\rho(s) \: \dvol_M) \: \pm \: \int_M \phi(s) \: \rho(s) \: \dvol_M$.
This modification will be crucial in later sections.
\end{remark}

Corollary \ref{cor1} was proven in \cite{Otto-Villani (2000)}.
The extension of Corollary \ref{cor1} to $P(M)$ was proven in
\cite{Cordero-Erausquin-McCann-Schmuckenschlaeger (2001)}.

We now give a slight refinement of the first variation result.

\begin{proposition} \label{prop4}
Under the assumptions of Proposition \ref{prop1},
\begin{align} \label{2.32}
\frac{dE}{dt} \Bigg|_{t = t_0} \: = \:
& \int_M \phi \: 
\left( \frac{\partial \rho}{\partial t} \: - \: \nabla^2 \rho \right)
\: \dvol_M \Bigg|_{s=0}^1 \: - \\
& \int_0^1
\int_M \left( \frac{\partial \phi}{\partial s} \: + \: \frac12 \:
|\nabla \phi|^2 \right) \: 
\left( \frac{\partial \rho}{\partial t} \: - \: \nabla^2 \rho \right) \: 
\dvol_M \: ds \: - \notag \\
& \int_0^1 \int_M \left(
|\Hess \phi|^2 \: + \: \Ric(\nabla \phi,\nabla \phi) \right)
\: \rho \: \dvol_M \: ds, \notag
\end{align}
where the right-hand side is evaluated at time $t = t_0$.
\end{proposition}
\begin{proof}
Integrating (\ref{2.29}) with respect to $s$ gives
\begin{align} \label{2.33}
- \: \int_M \phi \: \nabla^2 \rho \: \dvol_M \Bigg|_{s=0}^1 \: & = \:
- \: \int_0^1 \int_M \ \left( \frac{\partial \phi}{\partial s} \: + \:
\frac12 \: |\nabla \phi|^2 \right) \: \nabla^2 \rho \: \dvol_M \: ds \: + \\
& \: \: \: \: \: \: \int_0^1
\int_M \left(
|\Hess \phi|^2 \: + \: \Ric(\nabla \phi,\nabla \phi) \right)
\: \rho \: \dvol_M \: ds. \notag
\end{align}
The proposition follows from combining (\ref{2.12}) and (\ref{2.33}).
\end{proof}

\begin{corollary} \label{cor2} 
\cite{Otto (2001),Otto-Westdickenberg (2005),Sturm-von Renesse (2005)}
Suppose that
$\Ric(M, g) \ge 0$.
Let $e^{t \nabla^2}$ be the heat flow on $P^\infty(M)$.
Then for $\mu_0, \mu_1 \in P^\infty(M)$ and $t \ge 0$,
\begin{equation} \label{2.34}
W_2 \left( e^{t \nabla^2} \mu_0, e^{t \nabla^2} \mu_1 \right) \: \le \:
W_2(\mu_0, \mu_1).
\end{equation}
\end{corollary}
\begin{proof}
Using Theorem \ref{OW},
given $\epsilon > 0$, choose a smooth curve
$c \: : \: [0, 1] \rightarrow P^\infty(M)$ with $c(0) = \mu_0$ and
$c(1) = \mu_1$ so that
$E(c) \: \le \: \frac12 \: W_2(\mu_0, \mu_1)^2 \: + \: \epsilon$.
Define $c_t \: : \: [0,1] \rightarrow P^\infty(M)$ by
$c_t(s) \: = \: e^{t \nabla^2} c(s)$. By
Proposition \ref{prop4}, $E(c_t)$ is nonincreasing in $t$.
Hence $\frac12 \: W_2(c_t(0), c_t(1))^2 \: \le \: E(c_t) \: \le \: 
E(c_0) \: \le \: \frac12 \: W_2(\mu_0, \mu_1)^2 \: + \: \epsilon$. 
As $\epsilon$ was arbitrary, the corollary follows.
\end{proof}

We recall that $n = \dim(M)$. We now give a new convexity
result concerning Wasserstein geodesics.

\begin{proposition} \label{prop5}
If $\Ric(M,g) \ge 0$ then
$s {\mathcal E} \: + \: n s \log(s)$ is convex
along a Wasserstein geodesic in $P^\infty(M)$, defined for
$s \in [0, 1]$.
\end{proposition}
\begin{proof}
From (\ref{2.29}), $\frac{d^2 {\mathcal E}}{ds^2} \ge 0$.
As 
\begin{equation} \label{2.35}
\frac{d^2}{ds^2} \left( 
s {\mathcal E} \: + \: n s \log(s) \right) 
\: = \: s \: \frac{d^2 {\mathcal E}}{ds^2} \: + \:
2 \: \frac{d{\mathcal E}}{ds} \: + \: \frac{n}{s}, 
\end{equation}
it suffices to show that
\begin{equation} \label{2.36}
\left( \frac{d{\mathcal E}}{ds} \right)^2 \: \le \: n \: 
\frac{d^2 {\mathcal E}}{ds^2}.
\end{equation}
Now 
\begin{align} \label{2.37}
\left( \frac{d{\mathcal E}}{ds} \right)^2 \: & = \:
\left( \int_M 
\nabla^2 \phi \: \rho \: \dvol_M \right)^2 \: \le \:
\int_M 
(\nabla^2 \phi)^2 \: \rho \: \dvol_M \\
& \: \le \: n \: 
\int_M 
|\Hess \phi|^2 \: \rho \: \dvol_M \: \le \: n \: 
\frac{d^2 {\mathcal E}}{ds^2}, \notag
\end{align}
which proves the proposition.
\end{proof}

\begin{remark} \label{rem1}
More generally, suppose that a background measure
$\nu \: = \: e^{- \Psi} \: \dvol_M \in P^\infty(M)$ is such that
$(M, \nu)$ has $\Ric_N \ge 0$ in the sense of 
\cite[Definition 0.10]{Lott-Villani}. Recall the class of
functions $DC_\infty$ in \cite[Equation (0.5)]{Lott-Villani}. 
Given $U \in DC_\infty$, define $U_\nu \: : \:
P^\infty(M) \rightarrow \R$ as in 
\cite[Equation (0.1)]{Lott-Villani}.
Then using the calculations of
\cite[Appendix D]{Lott-Villani}, one can show that
$s U_\nu \: + \: N s \log(s)$ is convex
along a Wasserstein geodesic in $P^\infty(M)$.
\end{remark}

Now define ${\mathcal E} \: : \: P(M) \rightarrow \R \cup \{ \infty \}$
by
\begin{equation} \label{2.38}
{\mathcal E}(\mu) \: = \: 
\begin{cases}
\int_M \rho \: \log(\rho) \: \dvol_M & \text{ if } \mu \: = \: \rho \: 
\dvol_M, \\
\infty & \text{ if } \mu \text{ is not absolutely continuous with respect to }
\dvol_M.
\end{cases}
\end{equation}

\begin{proposition} \label{prop6}
If $\Ric(M,g) \ge 0$ then
$s {\mathcal E} \: + \: n s \log(s)$ is convex
along a Wasserstein geodesic in $P(M)$.
\end{proposition}
\begin{proof}
The proof uses the Lagrangian formulation of optimal transport;
see, for example, 
\cite[Pf. of Theorem 7.3]{Lott-Villani}.
We omit the details.
\end{proof}

\begin{remark} \label{rem2}
Similarly, in the setup of Remark \ref{rem1}, one has
that $s U_\nu \: + \: N s \log(s)$ is convex
along a Wasserstein geodesic in $P(M)$. It appears that most of
the results of \cite{Lott-Villani} could be derived using
the class of functions $DC_\infty$ and the functional
$s U_\nu \: + \: N s \log(s)$. The paper \cite{Lott-Villani}
used instead the class of functions $DC_N$ and the function
$U_\nu$.
\end{remark}

\section{Wasserstein distance and Ricci flow} \label{section3}

In this section we discuss a first monotonicity relation between
Ricci flow and optimal transport.  Namely, suppose that
the Ricci flow equation is satisfied and we have two solutions
$c_0(t), c_1(t)$ of the backward heat flow, acting on 
probability measures on $M$.
Then the Wasserstein distance $W_2(c_0(t), c_1(t))$ is nondecreasing
in $t$. We first give a quick formal proof. We then write out
a rigorous proof using the Otto calculus.
A proof using the Lagrangian approach appears in 
\cite{McCann-Topping}.

Let $(M, g(\cdot))$ be a solution to the Ricci flow equation
\begin{equation} \label{3.1}
\frac{dg}{dt} \: = \: - \: 2 \Ric.
\end{equation}
Then
\begin{equation} \label{3.2}
\frac{d(\dvol_M)}{dt} \: = \: - \: R \: \dvol_M.
\end{equation}

The metric $G$ on $P^\infty(M)$, from (\ref{2.3}), is also $t$-dependent.
Fix $\mu \in P^\infty(M)$ and $\delta \mu \in T_{\mu} P^\infty(M)$. At time
$t$, we can write $\mu \: = \: \rho \: \dvol_M$ and
$\delta \mu \: = \: V_{\phi}$ where $\rho$ and $\phi$ are $t$-dependent.

We now compute the first derivative of $G$ with respect to $t$.
\begin{proposition} \label{prop7}
\begin{equation} \label{3.3}
\frac{dG}{dt}( \delta \mu, \delta \mu) \: = \:
- \: 2 \int_M \Ric(\nabla \phi, \nabla \phi) \: d\mu.
\end{equation}
\end{proposition}
\begin{proof}
Letting $g^*$ denote the dual inner product on $T^* M$, we can write
\begin{equation} \label{3.4}
G( \delta \mu, \delta \mu) \: = \:
\int_M g^* (d \phi, d \phi) \: d\mu.
\end{equation}
Since the differential $d$ is invariantly defined, 
we have $\frac{d}{dt} d\phi = d \frac{d\phi}{dt}$.
Then
\begin{equation} \label{3.5}
\frac{dG}{dt}( \delta \mu, \delta \mu) \: = \:
2 \int_M \Ric(\nabla \phi, \nabla \phi) \: d\mu \: + \: 
2 \int_M g^* \left( d \phi, d \frac{d\phi}{dt} \right) \: d\mu.
\end{equation}
For any fixed $f \in C^\infty(M)$, we have
\begin{equation} \label{3.6}
\int_M f \: d(\delta \mu) \: = \: \int_M g^*(d f, d \phi) \: d\mu.
\end{equation}
Differentiating with respect to $t$ gives
\begin{equation} \label{3.7}
0 \: = \: 2 \int_M \Ric(\nabla f, \nabla \phi) \: d\mu \: + \:
\int_M g^* \left( d f, d \frac{d\phi}{dt} \right) \: d\mu.
\end{equation}
Putting $f = \phi$ gives
\begin{equation} \label{3.8}
0 \: = \: 2 \int_M \Ric(\nabla \phi, \nabla \phi) \: d\mu \: + \:
\int_M g^* \left( d \phi, d \frac{d\phi}{dt} \right) \: d\mu.
\end{equation}
Equation (\ref{3.3}) follows from combining (\ref{3.5}) and (\ref{3.8}).
\end{proof}

Let $\grad {\mathcal E}$ denote the formal gradient
of ${\mathcal E}$ on $P^\infty(M)$ and let $\Hess {\mathcal E}$
denote its Hessian.  Now the Lie derivative of the metric $G$ with respect to
the vector field $\grad {\mathcal E}$ is
${\mathcal L}_{\grad {\mathcal E}} G
\: = \: 2 \Hess {\mathcal E}$.
From Proposition \ref{prop3}, 
\begin{equation} \label{3.9}
(\Hess {\mathcal E})(V_\phi, V_\phi)  \: = \: \int_M \left(
|\Hess \phi|^2 \: + \: \Ric(\nabla \phi,\nabla \phi) \right)
\: \rho \: \dvol_M.
\end{equation}                 
Then from (\ref{3.3}) and (\ref{3.9}),
\begin{equation} \label{3.10}
\frac{dG}{dt} \: + \: {\mathcal L}_{\grad {\mathcal E}} G \: \ge \: 0.
\end{equation}

Let $\{ \phi_t \}$ be the $1$-parameter group generated by 
$\grad {\mathcal E}$. Equation (\ref{3.10}) implies that
$\phi_t^* G(t)$ is nondecreasing in $t$. In particular,
for any $\mu_0, \mu_1 \in P^\infty(M)$ the Wasserstein distance
$d_W(\phi_t(\mu_0), \phi_t(\mu_1))$ is nondecreasing in $t$.

It remains to compute the flow $\{ \phi_t \}$. This is a
well-known calculation.

\begin{lemma} \label{lem1}
In $T_{\rho \: \dvol_M} P^\infty(M)$,
\begin{equation} \label{3.11}
\grad {\mathcal E} \: = \: V_{\log \rho}.
\end{equation}
\end{lemma}
\begin{proof}
From (\ref{2.28}), for all $V_\phi \in T_{\rho \: \dvol_M} P^\infty(M)$, we have
\begin{align} \label{3.12}
G(V_\phi, \grad {\mathcal E})(\rho \: \dvol_M) \: & = \:
(V_\phi {\mathcal E})(\rho \: \dvol_M) \: = \:
\int_M \langle \nabla \phi, \nabla \rho \rangle \: \dvol_M \\
& = \:  
\int_M \langle \nabla \phi, \nabla \log \rho \rangle \: \rho \: \dvol_M
\: = \:  
G(V_\phi, V_{\log \rho})(\rho \: \dvol_M), \notag
\end{align}
from which the lemma follows.
\end{proof}

\begin{lemma} \label{lem2}
For $\mu \in P^\infty(M)$, if $\mu_t \: = \: \phi_t(\mu)$ then
\begin{equation} \label{3.13}
\frac{d \mu_t}{dt} \: = \: - \: \nabla^2 \mu_t.
\end{equation}
Equivalently, writing $\mu_t \: = \: \rho_t \: \dvol_M$, we have
\begin{equation} \label{3.14}
\frac{d \rho_t}{dt} \: = \: - \: \nabla^2 \rho_t \: + \: R \rho_t.
\end{equation}
\end{lemma}
\begin{proof}
Given $\mu \: = \: \rho \: \dvol_M$, we can write
\begin{equation} \label{3.15}
- \: \nabla^i (\rho \nabla_i \log \rho) \: \dvol_M \: = \:
- \: (\nabla^2 \rho) \: \dvol_M \: = \: - \: \nabla^2 \mu.
\end{equation}
Then (\ref{3.13}) follows from (\ref{3.11}) and (\ref{3.15}).
Equation (\ref{3.14}) follows from (\ref{3.2}).
\end{proof}

Thus we have formally shown that if $g(t)$ satisfies the
Ricci flow equation (\ref{3.1}) and
$\rho_{i,t}$ satisfies the backward heat equation
\begin{equation} \label{3.16}
\frac{d \rho_{i,t}}{dt} \: = \: - \: \nabla^2 
\rho_{i,t} \: + \: R \rho_{i,t}
\end{equation}
for $i \in \{0,1\}$ then the time-dependent Wasserstein
distance
$d_W(\rho_{0,t} \: \dvol_M, \rho_{1,t} \: \dvol_M)$ is 
nondecreasing in $t$. 

We now translate this into a rigorous proof using the Otto calculus.
We first derive a general formula for the derivative of
the energy functional $E$ along a $1$-parameter family of
smooth curves in $P^\infty(M)$.

\begin{proposition} \label{prop8}
Let $g(\cdot)$ solve the Ricci flow equation (\ref{3.1}) for
$t \in [t_0 - \epsilon, t_0 + \epsilon]$.
Let 
\begin{equation} \label{3.17}
\rho \: \dvol_M \: : \: [0,1] \times [t_0 - \epsilon, t_0 + \epsilon]
\rightarrow P^\infty(M)
\end{equation}
be a smooth map, with $\rho \equiv \rho(s,t)$.
Let 
\begin{equation} \label{3.18}
\phi \: : \: [0,1] \times [t_0 - \epsilon, t_0 + \epsilon]
\rightarrow C^\infty(M)
\end{equation}
be a smooth map that satisfies (\ref{2.6}), with
$\phi \equiv \phi(s,t)$.
Put 
\begin{equation} \label{3.19}
E(t) \: = \: \frac12 \int_0^1 \int_M |\nabla \phi|^2 \: \rho \: \dvol_M \: ds.
\end{equation}
Then 
\begin{align} \label{3.20}
\frac{dE}{dt} \Bigg|_{t = t_0} \: = \:
& \int_M \phi \left( \frac{\partial \rho}{\partial t} \: + 
 \: \nabla^2 \rho
\: - \: R \rho \right) \dvol_M \Bigg|_{s=0}^1 \: - \\
& \int_0^1 \int_M \left( \frac{\partial \phi}{\partial s} \: + \:
\frac12 \: |\nabla \phi|^2 \right) \: \left( \frac{\partial \rho}{\partial t}
\:  + 
 \: \nabla^2 \rho \: - \: R \rho \right)\: \dvol_M \: + \notag \\
& \int_0^1
\int_M |\Hess \phi|^2 \: \rho \: \dvol_M \: ds, \notag
\end{align}
where the right-hand side is evaluated at time $t = t_0$.
\end{proposition}
\begin{proof}
We have
\begin{equation} \label{3.21}
\frac{dE}{dt} \: = \: \int_0^1 \int_M \left( 
\Ric(\nabla \phi, \nabla \phi) \: \rho \: + \: 
\langle \nabla \phi,
\nabla \frac{\partial \phi}{\partial t} \rangle \: \rho \: + \: 
\frac12 \: |\nabla \phi|^2 \: \frac{\partial \rho}{\partial t} 
\: - \: \frac12 \: R |\nabla \phi|^2 \rho \right)
\: \dvol_M \: ds. 
\end{equation}
For a fixed $f \in C^\infty(M)$,
\begin{equation} \label{3.22}
\int_M f \: \frac{\partial \rho}{\partial s} \: \dvol_M \: = \:
\int_M \langle \nabla f, \nabla \phi \rangle \: \rho \: \dvol_M.
\end{equation}
Hence
\begin{align} \label{3.23}
& \int_M f \: \left(
\frac{\partial^2 \rho}{\partial s \partial t} \: - \:
R \frac{\partial \rho}{\partial s} \right) \: \dvol_M \: = \\
& \int_M \left( 
2 \: \Ric(\nabla f, \nabla \phi) \: \rho \: + \: 
\langle \nabla f, \nabla \frac{\partial \phi}{\partial t} 
\rangle \: \rho \: + \:
\langle \nabla f, \nabla \phi \rangle \: \frac{\partial \rho}{\partial t}
\: - \: R \: \langle \nabla f, \nabla \phi \rangle \: \rho
\right)
\: \dvol_M. \notag
\end{align}
Taking $f = \phi$ gives
\begin{align} \label{3.24}
& \int_M \phi \: \left(
\frac{\partial^2 \rho}{\partial s \partial t} \: - \:
R \frac{\partial \rho}{\partial s} \right) \: \dvol_M \: = \\
& \int_M \left(
2 \: \Ric(\nabla \phi, \nabla \phi) \: \rho \: + \: 
\langle \nabla \phi, \nabla \frac{\partial \phi}{\partial t} 
\rangle \: \rho \: + \:
|\nabla \phi|^2 \: \frac{\partial \rho}{\partial t}
\: - \: R \: |\nabla \phi|^2 \: \rho
\right)
\: \dvol_M. \notag
\end{align}
Equations (\ref{3.21}) and (\ref{3.24}) give
\begin{align} \label{3.25}
& \frac{dE}{dt} \: = \\
& \int_0^1 \int_M \left( 
\phi \: \frac{\partial^2 \rho}{\partial s \partial t}
\: - \: R  \phi \frac{\partial \rho}{\partial s}
 \: - \: \frac12 \:
| \nabla \phi|^2 \: \frac{\partial \rho}{\partial t}
\: - \: \Ric(\nabla \phi, \nabla \phi) \: \rho \: 
+ \: \frac12 \: R |\nabla \phi|^2
\rho
\right)
\: \dvol_M \: ds \: = \notag \\
& \int_0^1 \int_M \left( \frac{\partial}{\partial s} \left(
\phi \: \frac{\partial \rho}{\partial t} \right) 
\: - \: R  \phi \frac{\partial \rho}{\partial s}
\: - \:
\left(\frac{\partial \phi}{\partial s}  \: + \: \frac12 \: %
| \nabla \phi|^2 \right) \: \frac{\partial \rho}{\partial t}
\: - \: \Ric(\nabla \phi, \nabla \phi)  \: \rho
\: + \: \frac12 \: R |\nabla \phi|^2
\rho
\right) \notag \\
& \dvol_M \: ds = \notag \\
& \int_M 
\phi \: \frac{\partial \rho}{\partial t} \: \dvol_M \Bigg|_{s=0}^1
\: + \notag \\
& \int_0^1 \int_M \left( 
- \: R  \phi \frac{\partial \rho}{\partial s} \:
- \:
\left(\frac{\partial \phi}{\partial s}  \: + \: \frac12 \: %
| \nabla \phi|^2 \right) \: \frac{\partial \rho}{\partial t}
\: - \: \Ric(\nabla \phi, \nabla \phi) \: \rho 
\: + \: \frac12 \: R |\nabla \phi|^2
\rho
\right)
\: \dvol_M \: ds \notag
\end{align}
From (\ref{2.33}),
\begin{align} \label{3.26}
0 \: = \: \int_M 
\phi \: \nabla^2 \rho \: \dvol_M \Bigg|_{s=0}^1 \: + \:
& \int_0^1 \int_M \left[
|\Hess \phi|^2 \: + \: \Ric(\nabla \phi,\nabla \phi) \right]
\: \rho \: \dvol_M \: ds \: - \\
&
\int_0^1 \int_M \nabla^2 \rho \: \left( \frac{\partial \phi}{\partial s}
\: + \: \frac12 \: |\nabla \phi|^2 \right) \: \dvol_M \: ds. \notag
\end{align}
Finally, 
\begin{equation} \label{3.27}
\frac{\partial}{\partial s} \int_M R \phi \rho \: \dvol_M \: = \:
\int_M \left( R \frac{\partial \phi}{\partial s} \rho \: + \:
R \phi \frac{\partial \rho}{\partial s} \right)
 \: \dvol_M,    
\end{equation}
so
\begin{equation} \label{3.28}
0 \: = \: - \: \int_M R \phi \rho \: \dvol_M \Bigg|_{s=0}^1 \: + \:
\int_0^1 \int_M \left( R \frac{\partial \phi}{\partial s} \rho \: + \:
R \phi \frac{\partial \rho}{\partial s} \right)
 \: \dvol_M ds.    
\end{equation}
Adding (\ref{3.25}), (\ref{3.26}) and (\ref{3.28}) gives the proposition.
\end{proof}

\begin{corollary} \label{cor3}
For $i \in \{0,1\}$,
let $c_i(t)$ be a solution of the backward heat equation
(\ref{3.13}) in $P^\infty(M)$.
Then
$W_2(c_0(t), c_1(t))$ is nondecreasing in $t$.
\end{corollary}
\begin{proof}
Fix $t_0$. Using Theorem \ref{OW}, given $\epsilon > 0$, choose a smooth curve
$c \: : \: [0,1] \rightarrow P^\infty(M)$ so that
$c(0) \: = \: c_0(t_0)$, $c(1) \: = \: c_1(t_0)$ and
$E(c) \: \le \: \frac12 \: W_2(c_0(t_0), c_1(t_0))^2 \: + \: \epsilon$.
For $t \le t_0$, define $c_t \: : \: [0,1] \rightarrow P^\infty(M)$ by
saying that $c_{t_0}(s) = c(s)$ and $c_t(s)$ satisfies
equation (\ref{3.13}) in $t$. By Proposition \ref{prop8},
$E(c_t)$ is nondecreasing in $t$. Hence
$\frac12 \: W_2(c_0(t), c_1(t))^2 \le E(c_t) \le E(c_0) \le 
\frac12 \: W_2(c_0(t_0), c_1(t_0))^2 \: + \: \epsilon$.
Since $\epsilon$ was arbitrary, the corollary follows.
\end{proof}

\begin{remark} \label{rem3}
To see the relation between Corollary \ref{cor2} and Corollary \ref{cor3},
suppose that $M$ is Ricci flat, in which case the Ricci flow on $M$ is
constant.  Put $\tau = t_0 - t$. Then the backward heat
equation (\ref{3.13}) in $t$ becomes a forward heat equation in
$\tau$. Corollary \ref{cor2} says that the Wasserstein distance
between the heat flows
is nonincreasing in $\tau$, i.e. nondecreasing in $t$.
\end{remark}

Corollary \ref{cor3} was proven using Lagrangian methods in 
\cite{McCann-Topping}.

\section{Convexity of the ${\mathcal L}_0$-entropy}
\label{section4}

In this section we consider an analog ${\mathcal L}_0$ of
Perelman's ${\mathcal L}$-functional, which has the same
relationship to steady solitons as Perelman's ${\mathcal L}$-functional
has to shrinking solitons. Under a Ricci flow, we consider the
transport equation associated to the problem of minimizing the
${\mathcal L}_0$-cost.  We show the convexity of a modified
entropy functional.

Let $M$ be a connected closed manifold and let $g(\cdot)$ be a Ricci flow
solution on $M$. 

\begin{definition}
If $\gamma \: : \: [t^\prime, t^{\prime \prime}] \rightarrow M$ is a 
smooth curve then
its ${\mathcal L}_0$-length is
\begin{equation} \label{4.1}
{\mathcal L}_0(\gamma) \: = \: \frac12 \int_{t^\prime}^{t^{\prime \prime}}
\left( g \left( \frac{d\gamma}{dt}, \frac{d\gamma}{dt} \right) \: + \:
R(\gamma(t),t) \right) \: dt, 
\end{equation}
where the time-$t$ metric $g(t)$ is used to define the integrand.

Let $L_0^{t^\prime, t^{\prime \prime}}(m^\prime, m^{\prime \prime})$
be the infimum of ${\mathcal L}_0$ over curves $\gamma$ with
$\gamma(t^\prime) = m^\prime$ and
$\gamma(t^{\prime \prime}) = m^{\prime \prime}$.
\end{definition}

The Euler-Lagrange equation for the ${\mathcal L}_0$-functional
is easily derived to be
\begin{equation} \label{4.2}
\nabla_{\frac{d\gamma}{dt}} \left( \frac{d\gamma}{dt} \right) 
\: - \: \frac12 \: \nabla R
\: - \: 2 \: \Ric \left( \frac{d\gamma}{dt},\cdot \right) \: = \: 0.
\end{equation}
The ${\mathcal L}_0$-exponential map
${\mathcal L}_0\exp_{m^\prime}^{t^\prime, t^{\prime \prime}} \: : \:
T_{m^\prime} M \rightarrow M$
is defined by saying that
for $V \in T_{m^\prime}M$, one has
\begin{equation} \label{4.3}
{\mathcal L}_0\exp_{m^\prime}^{t^\prime, t^{\prime \prime}} 
\left(V \right) \: = \: \gamma(t^{\prime \prime})
\end{equation}
where $\gamma$ is the solution to (\ref{4.2}) with
$\gamma(t^\prime) = m^\prime$ and $\frac{d\gamma}{dt} \Big|_{t = t^\prime}
= V$.

\begin{definition}
Given $\mu^\prime, \mu^{\prime \prime} \in P(M)$, put
\begin{equation} \label{4.4}
C_0^{t^\prime, t^{\prime \prime}}
(\mu^\prime, \mu^{\prime \prime}) \: = \: \inf_\Pi
\int_{M \times M} 
L_0^{t^\prime, t^{\prime \prime}}(m^\prime, m^{\prime \prime}) \:
d\Pi(m^\prime, m^{\prime \prime}),
\end{equation}
where $\Pi$ ranges over the elements of $P(M \times M)$ whose pushforward
to $M$ under projection onto the first (resp. second) factor is
$\mu^\prime$ (resp. $\mu^{\prime \prime}$).
Given a continuous curve $c \: : \: [t^\prime, t^{\prime \prime}] \rightarrow
P(M)$, put
\begin{equation} \label{4.5}
{\mathcal A}_0(c) 
\: = \: \sup_{J \in \Z^+} \sup_{t^\prime = t_0 \le t_1 \le \ldots
\le t_J = t^{\prime \prime}} \sum_{j=1}^J C_0^{t_{j-1}, t_j}(c(t_{j-1}),
c(t_j)).
\end{equation}
\end{definition}

We can think of 
${\mathcal A}_0$ as a generalized energy functional associated
to the generalized metric $C_0$. 
By \cite[Theorem 7.21]{Villani2}, 
${\mathcal A}_0$ is a coercive action on $P(M)$
in the sense of
\cite[Definition 7.13]{Villani2}. In particular,
\begin{equation} \label{4.6}
C_0^{t^\prime, t^{\prime \prime}}
(\mu^\prime, \mu^{\prime \prime}) \: = \:
\inf_c {\mathcal A}_0(c),
\end{equation}
where $c$ ranges over continuous curves
$c \: : \: [t^\prime, t^{\prime \prime}] \rightarrow P(M)$
with $c(t^\prime) = \mu^\prime$ and
$c(t^{\prime \prime}) = \mu^{\prime \prime}$.

We now consider the equations that come from
minimizing the generalized energy functional ${\mathcal A}_0$,
when restricted to smooth curves in $P^\infty(M)$. 
If $c \: : [t_0, t_1] 
\rightarrow P^\infty(M)$ is a smooth curve in $P^\infty(M)$ then 
we write $c(t) \: = \: \rho(t) \: \dvol_M$ and let 
$\phi(t) \in C^\infty(M)$ satisfy
\begin{equation} \label{4.7}
\frac{\partial \rho}{dt} \: = \: - \: 
\nabla^i \left( \rho \nabla_i \phi \right) \: + \: R \rho.
\end{equation}
Note that $\phi(t)$ is uniquely defined up to an additive constant.
Using
(\ref{3.2}), the scalar curvature term in (\ref{4.7})  ensures that
\begin{equation} \label{4.8}
\frac{d}{dt} \int_M \rho \: \dvol_M \: = \: 0.
\end{equation}

Consider the Lagrangian 
\begin{equation} \label{4.9}
E_0(c) \: = \: \frac12 
\int_{t_0}^{t_1} \int_M \left( |\nabla \phi|^2
\: + \: R \right) \: \rho \: \dvol_M \: dt,
\end{equation}
where the integrand at time $t$ is computed using $g(t)$.

\begin{proposition} \label{prop9}
Let 
\begin{equation} \label{4.10}
\rho \: \dvol_M \: : \: [t_0, t_1] 
\times [- \epsilon, \epsilon]
\rightarrow P^\infty(M)
\end{equation}
be a smooth map, with
$\rho \equiv \rho(t,u)$.
Let 
\begin{equation} \label{4.11}
\phi \: : \: [t_0, t_1] \times [- \epsilon, \epsilon]
\rightarrow C^\infty(M)
\end{equation}
be a smooth map that satisfies (\ref{4.7}), with
$\phi \equiv \phi(t,u)$.
Then 
\begin{equation} \label{4.12}
\frac{dE_0}{du} \Bigg|_{u = 0} \: = \:
\int_M \phi \: \frac{\partial \rho}{\partial u}
\: \dvol_M \Bigg|_{t=t_0}^{t_1} \: - \: \
\int_{t_0}^{t_1}
\int_M \left( \frac{\partial \phi}{\partial t} \: + \: \frac12 \:
|\nabla \phi|^2  \: - \: \frac12 \: 
R \right) \: \frac{\partial \rho}{\partial u}
\: \dvol_M \: dt,
\end{equation}
where the right-hand side is evaluated at $u = 0$.
\end{proposition}
\begin{proof}
The proof is similar to that of Proposition \ref{prop1}. We omit the details.
\end{proof}

From (\ref{4.12}), the Euler-Lagrange equation for $E_0$ is 
\begin{equation} \label{4.13}
\frac{\partial \phi}{\partial t} \: = \: - \: \frac12 \: |\nabla \phi|^2
\: + \: \frac12 \: R \: + \: \alpha(t),
\end{equation}
where $\alpha \in C^\infty([t_0, t_1])$. Changing $\phi$ by a
spatially-constant function, we can assume that $\alpha = 0$, so
\begin{equation} \label{4.14}
\frac{\partial \phi}{\partial t} \: = \: - \: \frac12 \: |\nabla \phi|^2
\: + \: \frac12 \: R.
\end{equation}
If a smooth curve in $P^\infty(M)$ minimizes 
$E_0$, relative to its endpoints, then it will satisfy (\ref{4.14}).
For each $t_0  \le t^\prime < t^{\prime \prime} \le t_1$,
the viscosity solution of 
(\ref{4.14}) satisfies
\begin{equation} \label{4.15}
\phi(t^{\prime \prime})(m^{\prime \prime}) \: = \: 
\inf_{m^\prime \in M} \left( \phi(t^\prime)(m^\prime) \: + \:
L_0^{t^\prime, t^{\prime \prime}}(m^\prime, m^{\prime \prime})
\right). 
\end{equation}
Then the solution of (\ref{4.7}) satisfies 
\begin{equation} \label{4.16}
\rho(t^{\prime \prime}) \: \dvol_M \: = \: 
(F_{t^\prime, t^{\prime \prime}})_* (\rho(t^\prime) \: \dvol_M), 
\end{equation}
where the transport map
$F_{t^\prime, t^{\prime \prime}} \: : \: M \rightarrow M$ is given by
\begin{equation} \label{4.17}
F_{t^\prime, t^{\prime \prime}}(m^\prime) \: = \: 
{\mathcal L}_0\exp_{m^\prime}^{t^\prime, t^{\prime \prime}} 
\left( \nabla_{m^\prime}
\phi(t^\prime) \right).
\end{equation}

We now do certain calculations in an Otto calculus that
is adapted to the Ricci flow background.

\begin{proposition} \label{prop10}
Suppose that (\ref{4.7}) and (\ref{4.14}) are satisfied.
Then
\begin{equation} \label{4.18}
\frac{d}{dt} \int_M \phi \rho \: \dvol_M \: = \:
\frac12 \int_M \left( |\nabla \phi|^2 \: + \: R \right) \: \rho \: \dvol_M,
\end{equation}

\begin{equation} \label{4.19}
\frac12 \: \frac{d}{dt} \int_M |\nabla \phi|^2 \: 
\rho \: \dvol_M \: = \: \int_M \left( \Ric(\nabla \phi, \nabla \phi)
\: + \: \frac12 \: \langle \nabla R, \nabla \phi \rangle \right) \: \rho \: 
\dvol_M,
\end{equation}

\begin{equation} \label{4.20}
\frac{d}{dt} \int_M \rho \: \log(\rho) \: 
\dvol_M \: = \: \int_M \left( \langle \nabla \rho, \nabla \phi \rangle \: + \: 
R \: \rho \right) 
\dvol_M,
\end{equation}

\begin{equation} \label{4.21}
\frac{d}{dt} \int_M R \rho \: 
\dvol_M \: = \: \int_M \left( R_t \: + \: 
\langle \nabla R, \nabla \phi \rangle \right) \: \rho \: 
\dvol_M
\end{equation}

and
\begin{align} \label{4.22}
& \frac{d}{dt} \int_M \langle \nabla \rho, \nabla \phi \rangle \: 
\dvol_M \: = \\
& \int_M \left( |\Hess \phi|^2 \: + \: \Ric(\nabla \phi, \nabla \phi)
\: - \: 2 \: \langle \Ric, \Hess \phi \rangle \: - \: \frac12 
\: \nabla^2 R \right) \rho \: \dvol_M. \notag
\end{align}
\end{proposition}
\begin{proof}
For (\ref{4.18}),
\begin{align} \label{4.23}
& \frac{d}{dt} \int_M \phi \rho \: \dvol_M \: = \\
& \int_M \left( \left( - \: \frac12 \: |\nabla \phi|^2
\: + \: \frac12 \: R \right) \: \rho \: + \: \phi \left(
- \: 
\nabla^i \left( \rho \nabla_i \phi \right) \: + \: R \rho \right)
\: - \: R \phi \rho \right) \: \dvol_M \: = \notag \\
& \frac12 \int_M \left( |\nabla \phi|^2 \: + \: R \right) \: \rho \: \dvol_M.
\notag
\end{align}
For (\ref{4.19}),
\begin{align} \label{4.24}
& \frac12 \: \frac{d}{dt} \int_M |\nabla \phi|^2 \: 
\rho \: \dvol_M \: = \\
& \int_M \left( \Ric(\nabla \phi, \nabla \phi) \: \rho \: + \: 
\left\langle \nabla \phi, \nabla \left( - \: \frac12 \: |\nabla \phi|^2
\: + \: \frac12 \: R \right) \right\rangle\: \rho  \: + \right. \notag \\ 
& \left.
\: \: \: \: \: \: \: \: \: \: \: \: 
 \frac12 \: |\nabla \phi|^2 \:
\left(
- \: 
\nabla^i \left( \rho \nabla_i \phi \right) \: + \: R \rho \right) \: - \:
\frac12 \: R |\nabla \phi|^2 \: \rho  \right) \: \dvol_M \: = \notag \\
& \int_M \left( \Ric(\nabla \phi, \nabla \phi)
\: + \: \frac12 \: \langle \nabla R, \nabla \phi \rangle \right) \: \rho \: 
\dvol_M. \notag
\end{align}
For (\ref{4.20}),
\begin{align} \label{4.25}
& \frac{d}{dt} \int_M \rho \: \log(\rho) \: 
\dvol_M \: = \\
& \int_M \left( (\log(\rho) + 1) \left(
- \: 
\nabla^i \left( \rho \nabla_i \phi \right) \: + \: R \rho \right)
\: - \:  \rho \: \log(\rho) \: R \right) \: \dvol_M \: = \notag \\
& \int_M \left( \langle \nabla \rho, \nabla \phi \rangle \: + \: 
R \: \rho \right) 
\dvol_M. \notag
\end{align}
For (\ref{4.21}),
\begin{align} \label{4.26}
\frac{d}{dt} \int_M R \rho \: 
\dvol_M \: & = \: \int_M \left( R_t \rho \: + \: R
\left(
- \: 
\nabla^i \left( \rho \nabla_i \phi \right) \: + \: R \rho \right)
\: - \: R^2 \rho \right)
\: \dvol_M \\
& = 
\: \int_M \left( R_t \: + \: 
\langle \nabla R, \nabla \phi \rangle \right) \: \rho \: 
\dvol_M. \notag
\end{align}
For (\ref{4.22}),
\begin{align} \label{4.27}
& \frac{d}{dt} \int_M \langle \nabla \rho, \nabla \phi \rangle \: 
\dvol_M \: = \\ 
& \int_M \left( 2 \Ric(\nabla \rho, \nabla \phi) \: + \:
\left\langle \nabla \left( - \: 
\nabla^i \left( \rho \nabla_i \phi \right) \: + \: R \rho \right),
\nabla \phi \right\rangle
\: + \right. \notag \\
& \left. \: \: \: \: \: \: \: \: \: \: \: \:  
\left\langle \nabla \rho, \nabla \left( - \: \frac12 \: |\nabla \phi|^2
\: + \: \frac12 \: R \right) \right\rangle 
\: - \: R  \langle \nabla \rho, \nabla \phi \rangle \right) \:
\dvol_M. \notag
\end{align}
Now
\begin{equation} \label{4.28}
2 \int_M \Ric(\nabla \rho, \nabla \phi) \: \dvol_M \: = \: - \:
\int_M \left( \langle \nabla R, \nabla \phi \rangle \: + \: 
2 \: \langle \Ric, \Hess \phi \rangle \right) \: \rho \: \dvol_M 
\end{equation}
and
\begin{align} \label{4.29}
& \int_M \left( \left\langle \nabla \left(- \:  
\nabla^i \left( \rho \nabla_i \phi \right) \right),
\nabla \phi \right\rangle \: + \: 
\left\langle \nabla \rho, \nabla \left(- \: \frac12 \: |\nabla \phi|^2
 \right) \right\rangle \right) \: \dvol_M  \: = \\
& \int_M \left(- \: \langle \nabla \phi, \nabla (\nabla^2 \phi) \rangle
\: + \: \frac12 \: \nabla^2 |\nabla \phi|^2 \right) \: \rho \: \dvol_M \: = 
\notag \\
& \int_M \left( |\Hess \phi|^2 \: + \: \Ric(\nabla \phi, \nabla \phi) \right)
\: \rho \: \dvol_M. \notag
\end{align}
Thus
\begin{align} \label{4.30}
& \frac{d}{dt} \int_M \langle \nabla \rho, \nabla \phi \rangle \: 
\dvol_M \: = \\
& \int_M \left( 
|\Hess \phi|^2 \: + \: \Ric(\nabla \phi, \nabla \phi) 
\: - \: 2 \: \langle \Ric, \Hess \phi \rangle \right) \: \rho \: \dvol_M
\: + \notag \\ 
& \int_M \left( \: - \: 
\langle \nabla R, \nabla \phi \rangle \: \rho \: + \:
\langle \nabla(R \rho), \nabla \phi \rangle \: + \: 
\frac12 \: \langle \nabla \rho, \nabla R \rangle \: - \:
R \langle \nabla \rho, \nabla \phi \rangle \right) \: \dvol_M \: = \notag \\
& \int_M \left( |\Hess \phi|^2 \: + \: \Ric(\nabla \phi, \nabla \phi)
\: - \: 2 \: \langle \Ric, \Hess \phi \rangle \: - \: \frac12 
\: \nabla^2 R \right) \rho \: \dvol_M. \notag 
\end{align}
This proves the proposition.
\end{proof}

\begin{corollary} \label{cor4}
Under the hypotheses of Proposition \ref{prop10},
\begin{equation} \label{4.31}
\frac{d^2}{dt^2} \int_M \rho \: \log(\rho) \: \dvol_M \: = \:
\int_M \left( |\Ric \: - \: \Hess \phi|^2 \: + \:
\frac12 \: H(\nabla \phi) \right) \:
\rho \: \dvol_M,
\end{equation}
where
\begin{equation} \label{4.32}
H(X) \: = \: R_t \: + \: 2 \langle \nabla R, X \rangle \: + \:
2 \: \Ric(X,X)
\end{equation}
is Hamilton's trace Harnack expression.
Also,
\begin{equation} \label{4.33}
\frac{d^2}{dt^2} \int_M \left( \rho \: \log(\rho) \: - \: 
\phi \: \rho \right) \: \dvol_M \: = \:
\int_M |\Ric \: - \: \Hess \phi|^2 \:
\rho \: \dvol_M.
\end{equation}
In particular, $\int_M \left( \rho \: \log(\rho) \: - \: 
\phi \: \rho \right) \: \dvol_M$ is convex in $t$.
\end{corollary}
\begin{proof}
This follows from Proposition \ref{prop10}, along with the equation
\begin{equation} \label{4.34}
R_t \: = \: \nabla^2 R \: + \: 2 \: |\Ric|^2.
\end{equation}
\end{proof}

We now give the analog of Corollary \ref{cor4} for $P(M)$, using
results from \cite{Bernard-Buffoni (2007)} and
\cite[Chapters 7,10,13]{Villani2}. 
Let $c \: : \: [t_0, t_1] \rightarrow P(M)$ be a minimizing
curve for ${\mathcal A}_0$ relative to its endpoints, which
we assume to be absolutely continuous probability measures.
Then $c(t) \: = \: (F_{t_0, t})_* c(t_0)$, where there is a
semiconvex function $\phi_0 \in C(M)$ so that
$F_{t_0, t}(m_0) \: = \: {\mathcal L}_0 \exp_{m_0}^{t_0, t}
(\nabla_{m_0} \phi_0)$.
Define $\phi(t) \in C(M)$ by
\begin{equation} \label{4.35}
\phi(t)(m) \: = \: 
\inf_{m_0 \in M} \left( \phi_0(m_0) \: + \:
L_0^{t_0, t}(m_0, m) \right).
\end{equation}
Define ${\mathcal E} \: : \: P(M) \rightarrow \R \cup \{ \infty \}$
as in (\ref{2.38}).
\begin{proposition} \label{prop11}
${\mathcal E}(c(t)) \: - \: \int_M \phi(t) \: dc(t)$ is convex in $t$.
\end{proposition}
\begin{proof}
The proof is along the lines of the proof of Proposition \ref{prop16} ahead.
\end{proof}

\begin{remark}
The function $\phi$ also enters as a solution of the dual
Kantorovitch problem.  See \cite[Theorem 7.36]{Villani2} (where
what we call $\phi$ is called $\psi$).
\end{remark}

\begin{remark}
Suppose that the Ricci flow solution $(M, g(\cdot))$ is a gradient
steady soliton, meaning that it is a Ricci flow solution with
$\Ric \: + \: \Hess(f) \: = \: 0$,
where $f$ satisfies $\frac{\partial f}{\partial t} \: = \: 
|\nabla f|^2$. Differentiating spatially and temporally, one shows that
$|\nabla f|^2 \: + \: R \: = \: C$ for some constant
$C$. Then there is a solution of (\ref{4.13}) with $\phi \: = \: - \: f$
and $\alpha \: = \: - \: \frac12 \: C$. If $\rho$ is transported
along the static vector field $- \nabla f$, i.e. satisfies (\ref{4.7}), 
then (\ref{4.33}) says that 
$\int_M (\rho \: \log(\rho) \: + \: f \: \rho) \: \dvol_M$ is
linear in $t$. 
\end{remark}

\section{Monotonicity of the ${\mathcal L}_0$-cost under a
backward heat flow}
\label{section5}

In this section we discuss the ${\mathcal L}_0$-cost between two
measures that each evolve under the backward heat flow.
The results are analogs of results of Topping
for the ${\mathcal L}$-cost \cite{Topping}.  We first
compute the variation of $E_0$ with respect to
a one-parameter family of curves that begin and end at shifted times.
We use this to show, within the Otto calculus, that if 
measures $c^\prime(\cdot)$ and $c^{\prime \prime}(\cdot)$ 
evolve under the backward heat flow then the 
${\mathcal L}_0$-cost between $c^\prime(t^\prime + u)$ and 
$c^{\prime \prime}(t^{\prime \prime} + u)$ is nondecreasing in $u$.
We then show that this implies the monotonicity of Perelman's
${\mathcal F}$-functional, in analogy to what Topping did for
Perelman's ${\mathcal W}$-functional.

\begin{proposition} \label{prop12}
Take $t_0 < t^\prime < t^{\prime \prime} < t_1$.
For small $\epsilon$, suppose that
$c \: : \: [t^\prime, t^{\prime \prime}] \times (- \epsilon, \epsilon) 
\rightarrow P^\infty(M)$ is
a smooth map, where $c \equiv c(t,u)$.
Define $c_u \: : \: 
[t^\prime + u, t^{\prime \prime} + u] \rightarrow P^\infty(M)$
by $c_u(t) \: = \: c(t-u, u)$.
Put $\mu^\prime = c_0(t^\prime)$ and 
$\mu^{\prime \prime} = c_0(t^{\prime \prime})$.
Suppose that $c_0$ is a minimizer for $E_0$ among curves from 
$[t^\prime, t^{\prime \prime}]$ to $P^\infty(M)$ whose endpoints are
$\mu^\prime$ and $\mu^{\prime \prime}$.
Put $V(t) \: = \: \frac{\partial c}{\partial u} \Big|_{u=0}$.
Then
\begin{equation} \label{5.1}
\frac{dE_0(c_u)}{du} \Bigg|_{u = 0} \: = 
\int_{t^\prime}^{t^{\prime \prime}}  \int_M
| \Ric \: - \: \Hess \phi |^2 \: \rho \: \dvol_M \: dt
\: + \:
\int_M \phi(t) \: \left( V(t) \: + \: 
\nabla^2 \rho \: \dvol_M
\right) 
\Bigg|_{t=t^\prime}^{t^{\prime \prime}}.
\end{equation}
\end{proposition}
\begin{proof}
For any $u \in (-\epsilon, \epsilon)$, we can write
\begin{equation} \label{5.2}
E_0(c_u) \: = \: \frac12 \int_{t^\prime}^{t^{\prime \prime}} 
\left( G \left( \frac{\partial c}{\partial t}, 
\frac{\partial c}{\partial t} \right) \: + \:
\int_M R \: c(t,u) \right) \: dt,
\end{equation}
where the integrand is evaluated using the metric at time $t+u$, and
the $c(t,u)$ in the term $R \: c(t,u)$ is taken to be a measure on $M$.
There is a well-defined notion of covariant derivative on $P^\infty(M)$
\cite[Proposition 2]{Lott (2008)}.
Letting $D$ denote directional covariant differentiation on $P^\infty(M)$,
\begin{align} \label{5.3}
& \frac{dE_0(c_u)}{du} \Bigg|_{u = 0} \: = \\
& \int_{t^\prime}^{t^{\prime \prime}} 
\left( \frac12 \: 
G_t \left( \frac{dc_0}{dt}, \frac{dc_0}{dt} \right) \: + \:
G \left( \frac{dc_0}{dt}, D_{\frac{dc_0}{dt}} V(t) 
\right) \: + \:
\frac12 \: \int_M R_t c_0(t) \: + \:
\frac12 \: \int_M R V(t) \right) \: dt \: = \notag \\ 
&
\int_{t^\prime}^{t^{\prime \prime}} 
\left( \frac12 \: 
G_t \left( \frac{dc_0}{dt}, \frac{dc_0}{dt} \right) \: + \:
G \left( D_{\frac{dc_0}{dt}} 
\left( \frac{dc_0}{dt} \right), V(t) 
\right) \: + \:
\frac12 \: \int_M R_t c_0(t) \: + \:
\frac12 \: \int_M R V(t) \right) \: dt \: + \notag \\
& G \left( 
\frac{dc_0}{dt}, V(t) 
\right) \Bigg|_{t=t^\prime}^{t^{\prime \prime}}. \notag
\end{align}
As $c_0$ is a minimizer,
\begin{equation} \label{5.4}
\frac{dE_0(c_u)}{du} \Bigg|_{u = 0} \: = 
\int_{t^\prime}^{t^{\prime \prime}} 
\left( \frac12 \: 
G_t \left( \frac{dc_0}{dt}, \frac{dc_0}{dt} \right) \: + \:
\frac12 \: \int_M R_t c_0(t) \right) \: dt \: + \:
G \left( 
\frac{dc_0}{dt}, V(t) 
\right) \Bigg|_{t=t^\prime}^{t^{\prime \prime}}.
\end{equation}

For any $f \in C^\infty(M)$,
\begin{equation} \label{5.5}
\frac{d}{dt} \int_M f \: \rho \: \dvol_M \: = \: 
\int_M \langle \nabla f, \nabla \phi \rangle \: \rho \: \dvol_M.
\end{equation}
This gives $\frac{dc_0}{dt}$ in terms of $\phi$.
Then from (\ref{2.4}) and Proposition \ref{prop7},
\begin{equation} \label{5.6}
\frac{dE_0(c_u)}{du} \Bigg|_{u = 0} \: = 
\int_{t^\prime}^{t^{\prime \prime}} \int_M
\left( - \:
\Ric(\nabla \phi, \nabla \phi) \: \rho \: \dvol_M \: + \:
\frac12 \: \int_M R_t c_0(t) \right) \: dt \: + \:
\int_M \phi(t) \: V(t) \Bigg|_{t=t^\prime}^{t^{\prime \prime}}.
\end{equation}
From (\ref{4.30}),
\begin{align} \label{5.7}
& \int_M \langle \nabla \rho, \nabla \phi \rangle \: 
\dvol_M \Bigg|_{t^\prime}^{t^{\prime \prime}} \: = \\
& \int_{t^\prime}^{t^{\prime \prime}}
\int_M \left( |\Hess \phi|^2 \: + \: \Ric(\nabla \phi, \nabla \phi)
\: - \: 2 \: \langle \Ric, \Hess \phi \rangle \: - \: \frac12 
\: \nabla^2 R \right) \rho \: \dvol_M \: dt. \notag
\end{align}
The proposition follows from the curvature evolution equation (\ref{4.34}),
(\ref{5.6}) and (\ref{5.7}).
\end{proof}

\begin{corollary} \label{cor5}
Under the hypotheses of Proposition \ref{prop12}, suppose that
each $c_u$ is a minimizer for $E_0$ relative to its endpoints.
Suppose that the endpoint 
measures $c_u(t^\prime+u) = c(t^\prime, u)$ and 
$c_u(t^{\prime \prime}+u) = c(t^{\prime \prime}, u)$ each
satisfy the backward heat equation, in the variable $u$ :
\begin{equation} \label{5.8}
\frac{dc}{du} \: = \: - \: \nabla^2 c.
\end{equation}
Then $C_0^{t^\prime+u, t^{\prime \prime}+u}(c_u(t^\prime+u),
c_u(t^{\prime \prime}+u))$ is 
nondecreasing in $u$.
\end{corollary}

We now give the general statement about the monotonicity of the
${\mathcal L}_0$-cost for two measures that evolve under the
backward heat flow,  without the extra assumption in Corollary \ref{cor5}
that minimizers $c_u$ stay in $P^\infty(M)$.
Its proof is an analog of Topping's proof
of the corresponding statement for the ${\mathcal L}$-cost
\cite{Topping}.

\begin{proposition} \label{prop13}
Suppose that $c^\prime \: : \: [t_0, t_1] \rightarrow P^\infty(M)$ and 
$c^{\prime \prime} \: : \: [t_0, t_1] \rightarrow P^\infty(M)$
satisfy (\ref{3.13}). 
Then $C_0^{t^\prime+u, t^{\prime \prime}+u}(c^\prime(t^\prime+u),
c^{\prime \prime}(t^{\prime \prime}+u))$ is 
nondecreasing in $u$.
\end{proposition}

Using Proposition \ref{prop13}, we now reprove the fact that Perelman's
${\mathcal F}$-functional is monotonic
\cite{Perelman}.  The proof is along the
lines of Topping's proof \cite{Topping} of the corresponding result for
Perelman's ${\mathcal W}$-functional.

\begin{corollary} \label{cor6}
Suppose that $\alpha \: : \: [t_0, t_1] \rightarrow P^\infty(M)$ is a solution
of (\ref{3.13}).
Write $\alpha(t) \: = \: \rho(t) \: \dvol_M$. Then
\begin{equation} \label{5.9}
{\mathcal F} \: = \: \int_M \left( |\nabla \log(\rho)|^2 \: + \: R \right)
\: \rho \: \dvol_M
\end{equation}
is nondecreasing in $t$.
\end{corollary}
\begin{proof}
Put $c^\prime \: = \: c^{\prime \prime} \: = \: \alpha$.
Take $t^{\prime \prime} > t^\prime$.
By Corollary \ref{cor5},
if $u > 0$ then
$C_0^{t^\prime+u, t^{\prime \prime}+u}(\alpha(t^\prime + u),
\alpha(t^{\prime \prime} + u)) \: \ge \: 
C_0^{t^\prime, t^{\prime \prime}}(\alpha(t^\prime),
\alpha(t^{\prime \prime}))$, so
\begin{equation} \label{5.10}
\frac{C_0^{t^\prime+u, t^{\prime \prime}+u}(\alpha(t^\prime + u),
\alpha(t^{\prime \prime} + u))}{t^{\prime \prime} - t^\prime} \: \ge \: 
\frac{
C_0^{t^\prime, t^{\prime \prime}}(\alpha(t^\prime),
\alpha(t^{\prime \prime}))}{t^{\prime \prime} - t^\prime}.
\end{equation}
From (\ref{4.4}),
\begin{equation} \label{5.11}
\lim_{t^{\prime \prime} \rightarrow t^\prime} 
\frac{1}{t^{\prime \prime} - t^\prime} \:
C_0^{t^\prime, t^{\prime \prime}}(\alpha(t^\prime),
\alpha(t^{\prime \prime})) \: = \: \frac12
\int_M \left( |\nabla \phi|^2 \: + \: R \right) \: \rho \: \dvol_M,
\end{equation}
where $\phi$ satisfies (\ref{4.7}) and the right-hand side is
evaluated at time $t^\prime$.
As $\rho$ satisfies (\ref{3.14}), we can take $\phi \: = \: \log(\rho)$.
The corollary follows.
\end{proof}

\section{Convexity of the ${\mathcal L}_-$-entropy}
\label{section6}

In this section we extend the results of Section \ref{section4} from the
${\mathcal L}_0$-functional to the ${\mathcal L}_-$-functional.
Optimal transport with an ${\mathcal L}_-$-cost was considered
in \cite{Topping}.
As the results of this section 
are analogs of those in Section \ref{section4}, we only
indicate the needed changes.

Let $M$ be a connected closed manifold and let $g(\cdot)$ be a Ricci flow
solution on $M$. We put $\tau = t_0 - t$ and write the Ricci 
flow equation in terms of $\tau$, i.e.
\begin{equation} \label{6.1}
\frac{dg}{d\tau} \: = \: 2 \Ric(g(\tau)).
\end{equation}

\begin{definition}
If $\gamma \: : \: [\tau^\prime, \tau^{\prime \prime}] \rightarrow M$ is a 
smooth curve with $\tau^\prime > 0$ then
its ${\mathcal L}_-$-length is
\begin{equation} \label{6.2}
{\mathcal L}_-(\gamma) \: = \: \frac12 \int_{\tau^\prime}^{\tau^{\prime \prime}}
\sqrt{\tau} \:
\left( g \left( \frac{d\gamma}{d\tau}, \frac{d\gamma}{d\tau} \right) \: + \:
R(\gamma(\tau),\tau) \right) \: d\tau, 
\end{equation}
where the time-$\tau$ metric $g(\tau)$ is used to define the integrand.

Let $L_-^{\tau^\prime, \tau^{\prime \prime}}(m^\prime, m^{\prime \prime})$
be the infimum of ${\mathcal L}_-$ over curves $\gamma$ with
$\gamma(\tau^\prime) = m^\prime$ and
$\gamma(\tau^{\prime \prime}) = m^{\prime \prime}$.
\end{definition}

The Euler-Lagrange equation for the ${\mathcal L}_-$-functional
is easily derived \cite[(7.2)]{Perelman} to be
\begin{equation} \label{6.3}
\nabla_{\frac{d\gamma}{d\tau}} \left( \frac{d\gamma}{d\tau} \right) 
\: - \: \frac12 \: \nabla R
\: + \: \frac{1}{2\tau} \: \frac{d\gamma}{d\tau} 
\: + \: 2 \: \Ric \left( \frac{d\gamma}{d\tau},\cdot \right) \: = \: 0.
\end{equation}
The ${\mathcal L}_-$-exponential map 
${\mathcal L}_-\exp_{m^\prime}^{\tau^\prime, \tau^{\prime \prime}} \: : \:
T_{m^\prime} M \rightarrow M$
is defined by saying that
for $V \in T_{m^\prime}M$, one has
\begin{equation} \label{6.4}
{\mathcal L}_-\exp_{m^\prime}^{\tau^\prime, \tau^{\prime \prime}} 
\left(V \right) \: = \: \gamma(\tau^{\prime \prime})
\end{equation}
where $\gamma$ is the solution to (\ref{6.3}) with
$\gamma(\tau^\prime) = m^\prime$ and $\frac{d\gamma}{d\tau} \Big|_{\tau = \tau^\prime}
= V$. Note that our ${\mathcal L}_-$-exponential map differs slightly from
Perelman's ${\mathcal L}$-exponential map.

\begin{definition}
Given $\mu^\prime, \mu^{\prime \prime} \in P(M)$, put
\begin{equation} \label{6.5}
C_-^{\tau^\prime, \tau^{\prime \prime}}
(\mu^\prime, \mu^{\prime \prime}) \: = \: \inf_\Pi
\int_{M \times M} 
L_-^{\tau^\prime, \tau^{\prime \prime}}(m^\prime, m^{\prime \prime}) \:
d\Pi(m^\prime, m^{\prime \prime}),
\end{equation}
where $\Pi$ ranges over the elements of $P(M \times M)$ whose pushforward
to $M$ under projection onto the first (resp. second) factor is
$\mu^\prime$ (resp. $\mu^{\prime \prime}$).
Given a continuous curve 
$c \: : \: [\tau^\prime, \tau^{\prime \prime}] \rightarrow P(M)$, put
\begin{equation} \label{6.6}
{\mathcal A}_-(c) 
\: = \: \sup_{J \in \Z^+} \sup_{\tau^\prime = \tau_0 \le \tau_1 \le \ldots
\le \tau_J = \tau^{\prime \prime}} \sum_{j=1}^J C_-^{\tau_{j-1}, \tau_j}(c(\tau_{j-1}),
c(\tau_j)).
\end{equation}
\end{definition}

We can think of 
${\mathcal A}_-$ as a generalized length functional associated
to the generalized metric $C_-$. 
By \cite[Theorem 7.21]{Villani2}, 
${\mathcal A}_-$ is a coercive action on $P(M)$
in the sense of
\cite[Definition 7.13]{Villani2}. In particular,
\begin{equation} \label{6.7}
C_-^{\tau^\prime, \tau^{\prime \prime}}
(\mu^\prime, \mu^{\prime \prime}) \: = \:
\inf_c {\mathcal A}_-(c),
\end{equation}
where $c$ ranges over continuous curves 
$c \: : \: [\tau^\prime, \tau^{\prime \prime}] \rightarrow
P(M)$ with $c(\tau^\prime) = \mu^\prime$ and
$c(\tau^{\prime \prime}) = \mu^{\prime \prime}$.

If $c \: : [\tau_0, \tau_1] 
\rightarrow P^\infty(M)$ is a smooth curve in $P^\infty(M)$, with
$\tau_0 > 0$, then 
we write $c(\tau) \: = \: \rho(\tau) \: \dvol_M$ and let $\phi(\tau)$ satisfy
\begin{equation} \label{6.8}
\frac{\partial \rho}{d\tau} \: = \: - \: 
\nabla^i \left( \rho \nabla_i \phi \right) \: - \: R \rho.
\end{equation}
Note that $\phi(\tau)$ is uniquely defined up to an additive constant.
The scalar curvature term in (\ref{6.8})  ensures that
\begin{equation} \label{6.9}
\frac{d}{d\tau} \int_M \rho \: \dvol_M \: = \: 0.
\end{equation}

Consider the Lagrangian 
\begin{equation} \label{6.10}
E_-(c) \: = \: 
\int_{\tau_0}^{\tau_1} \int_M \sqrt{\tau} \left( |\nabla \phi|^2
\: + \: R \right) \: \rho \: \dvol_M \: d\tau,
\end{equation}
where the integrand at time $\tau$ is computed using $g(\tau)$.

\begin{proposition} \label{prop14}
Let 
\begin{equation} \label{6.11}
\rho \: \dvol_M \: : \: [\tau_0, \tau_1] 
\times [- \epsilon, \epsilon]
\rightarrow P^\infty(M)
\end{equation}
be a smooth map, with $\rho \equiv \rho(\tau,u)$.
Let 
\begin{equation} \label{6.12}
\phi \: : \: [\tau_0, \tau_1] \times [- \epsilon, \epsilon]
\rightarrow C^\infty(M)
\end{equation}
be a smooth map that satisfies (\ref{6.8}), with
$\phi \equiv \phi(\tau,u)$.
Then 
\begin{align} \label{6.13}
\frac{dE_-}{du} \Bigg|_{u = 0} \: = \:
& 2 \sqrt{\tau} \int_M \phi \: \frac{\partial \rho}{\partial u}
\: \dvol_M \Bigg|_{\tau=\tau_0}^{\tau_1} \: - \\
& 2 \: 
\int_{\tau_0}^{\tau_1}
\int_M 
\sqrt{\tau} \: \left( \frac{\partial \phi}{\partial \tau} \: + \: \frac12 \:
|\nabla \phi|^2  \: - \: \frac12 \: 
R \: + \: \frac{1}{2\tau} \: \phi \right) \: \frac{\partial \rho}{\partial u}
\: \dvol_M \: d\tau, \notag
\end{align}
where the right-hand side is evaluated at $u = 0$.
\end{proposition}
\begin{proof}
The proof is similar to that of Proposition \ref{prop9}. We omit the details.
\end{proof}

From (\ref{6.13}), the Euler-Lagrange equation for $E_-$ is 
\begin{equation} \label{6.14}
\frac{\partial \phi}{\partial \tau} \: = \: - \: \frac12 \: |\nabla \phi|^2
\: + \: \frac12 \: R \: - \: \frac{1}{2\tau} \: \phi \: + \: \alpha(\tau),
\end{equation}
where $\alpha \in C^\infty([\tau_0, \tau_1])$. Changing $\phi$ by a
spatially-constant function, we can assume that $\alpha = 0$, so
\begin{equation} \label{6.15}
\frac{\partial \phi}{\partial \tau} \: = \: - \: \frac12 \: |\nabla \phi|^2
\: + \: \frac12 \: R \: - \: \frac{1}{2\tau} \: \phi.
\end{equation}
If a smooth curve in $P^\infty(M)$ minimizes 
$E_-$, relative to its endpoints, then it will satisfy (\ref{6.15}).
For each $\tau_0  \le \tau^\prime < \tau^{\prime \prime} \le \tau_1$,
the viscosity solution of 
(\ref{6.15}) satisfies
\begin{equation} \label{6.16}
2 \sqrt{\tau^{\prime \prime}} \:
\phi(\tau^{\prime \prime})(m^{\prime \prime}) \: = \: 
\inf_{m^\prime \in M} \left( 
2 \sqrt{\tau^\prime} \: \phi(\tau^\prime)(m^\prime) \: + \:
L_-^{\tau^\prime, \tau^{\prime \prime}}(m^\prime, m^{\prime \prime})
\right). 
\end{equation}
Then the solution of (\ref{6.8}) satisfies 
\begin{equation} \label{6.17}
\rho(\tau^{\prime \prime}) \: \dvol_M \: = \: 
(F_{\tau^\prime, \tau^{\prime \prime}})_* (\rho(\tau^\prime) \: \dvol_M), 
\end{equation}
where the transport map
$F_{\tau^\prime, \tau^{\prime \prime}} \: : \: M \rightarrow M$ is given by
\begin{equation} \label{6.18}
F_{\tau^\prime, \tau^{\prime \prime}}(m^\prime) \: = \: 
{\mathcal L}_-\exp_{m^\prime}^{\tau^\prime, \tau^{\prime \prime}} 
\left( \nabla_{m^\prime}
\phi(\tau^\prime) \right).
\end{equation}
Our function $\phi$ is related to the function $\varphi$ of
\cite{Topping} by $\phi \: = \: - \: \frac{\varphi}{2 \sqrt{\tau}}$.

\begin{proposition} \label{prop15}
Suppose that (\ref{6.8}) and (\ref{6.15}) are satisfied.
Then
\begin{equation} \label{6.19}
\frac{d}{d\tau} \int_M \phi \rho \: \dvol_M \: = \:
\frac12 \int_M \left( |\nabla \phi|^2 \: + \: R \right) \: \rho \: \dvol_M
\: - \: \frac{1}{2\tau} \int_M \phi \rho \: \dvol_M,
\end{equation}

\begin{align} \label{6.20}
\frac12 \: \frac{d}{d\tau} \int_M |\nabla \phi|^2 \: 
\rho \: \dvol_M \: = \: & \int_M \left(- \: \Ric(\nabla \phi, \nabla \phi)
\: + \: \frac12 \: \langle \nabla R, \nabla \phi \rangle \right) \: \rho \: 
\dvol_M \: - \\
& \frac{1}{2\tau} \: \int_M |\nabla \phi|^2 \: 
\rho \: \dvol_M, \notag
\end{align}

\begin{equation} \label{6.21}
\frac{d}{d\tau} \int_M \rho \: \log(\rho) \: 
\dvol_M \: = \: \int_M \left( \langle \nabla \rho, \nabla \phi \rangle \: - \: 
R \: \rho \right) 
\dvol_M,
\end{equation}

\begin{equation} \label{6.22}
\frac{d}{d\tau} \int_M R \rho \: 
\dvol_M \: = \: \int_M \left( R_\tau \: + \: 
\langle \nabla R, \nabla \phi \rangle \right) \: \rho \: 
\dvol_M
\end{equation}

and
\begin{align} \label{6.23}
& \frac{d}{d\tau} \int_M \langle \nabla \rho, \nabla \phi \rangle \: 
\dvol_M \: = \\
& \int_M \left( |\Hess \phi|^2 \: + \: \Ric(\nabla \phi, \nabla \phi)
\: + \: 2 \: \langle \Ric, \Hess \phi \rangle \: - \: \frac12 
\: \nabla^2 R \right) \rho \: \dvol_M \: - \notag \\
& \frac{1}{2\tau} \: \int_M \langle \nabla \rho, \nabla \phi \rangle \: 
\dvol_M. \notag
\end{align}
\end{proposition}
\begin{proof}
The proof is similar to that of Proposition \ref{prop10}. We omit the details.
\end{proof}

\begin{corollary} \label{cor7}
Under the hypotheses of Proposition \ref{prop15},
\begin{equation} \label{6.24}
\left( \tau^{\frac12} \: \frac{d}{d\tau} \right)^2 
\int_M \rho \: \log(\rho) \: \dvol_M \: = \:
\tau \int_M \left( |\Ric \: + \: \Hess \phi|^2 \: + \:
\frac12 \: H(\nabla \phi) \right) \:
\rho \: \dvol_M,
\end{equation}
where
\begin{equation} \label{6.25}
H(X) \: = \: - \: R_\tau \: - \: 2 \langle \nabla R, X \rangle \: + \:
2 \: \Ric(X,X) \: - \: \frac{R}{\tau}
\end{equation}
is Hamilton's trace Harnack expression.
Also,
\begin{align} \label{6.26}
& \left( \tau^{\frac32} \: \frac{d}{d\tau} \right)^2 
\left(  \int_M (\rho \: \log(\rho) \: + \: 
\phi \: \rho ) \: \dvol_M \: + \: \frac{n}{2} \: \log(\tau) \right) \: = 
\\
&
\tau^3 \: \int_M \left| \Ric \: + \: \Hess \phi \: - \: \frac{g}{2\tau}
\right|^2 \:
\rho \: \dvol_M. \notag
\end{align}
In particular, $\int_M \left( \rho \: \log(\rho) \: + \: 
\phi \: \rho \right) \: \dvol_M \: + \: 
\frac{n}{2} \: \log(\tau)$ is convex in $\tau^{- \: \frac12}$.
\end{corollary}
\begin{proof}
This follows from Proposition \ref{prop15}, along with the equation
\begin{equation} \label{6.27}
R_\tau \: = \: - \: \nabla^2 R \: - \: 2 \: |\Ric|^2,
\end{equation}
after some calculations.
\end{proof}

\begin{remark} \label{addedremark}
In \cite{Topping} it is shown, for transport in $P(M)$ between
two elements of $P^\infty(M)$, that
\begin{equation}
\left( \tau^{\frac12} \: \frac{d}{d\tau} \right)^2 
\int_M \rho \: \log(\rho) \: \dvol_M \: \ge \:
\frac12 \: \tau \int_M \
H(\nabla \phi) \:
\rho \: \dvol_M
\end{equation}
and
\begin{equation}
\left( \tau^{\frac32} \: \frac{d}{d\tau} \right)^2 
\int_M \rho \: \log(\rho) \: \dvol_M \: \ge \:
\frac12 \: \tau^3 \: \int_M H(\nabla \phi) \:
\rho \: \dvol_M \: - \: \frac{n}{4} \: \tau.
\end{equation}
\end{remark}

\section{Monotonicity of the reduced volume}
\label{section7}

In this section we give the extension of Corollary \ref{cor7} to
$P(M)$. We then reprove the monotonicity of
Perelman's reduced volume \cite{Perelman}.

Let $c \: : \: [\tau_0, \tau_1] \rightarrow P(M)$ be a minimizing
curve for ${\mathcal A}_-$ relative to its endpoints. We
assume that $c(\tau_0)$ are $c(\tau_1)$ are 
absolutely continuous with respect to 
a Riemannian volume density on $M$.
Then $c(\tau) \: = \: (F_{\tau_0, \tau})_* c(\tau_0)$, where there is a
semiconvex function $\phi_0 \in C(M)$ so that
$F_{\tau_0, \tau}(m_0) \: = \: {\mathcal L}_- \exp_{m_0}^{\tau_0, \tau}
(\nabla_{m_0} \phi_0)$ \cite{Bernard-Buffoni (2007)},
\cite[Chapters 10,13]{Villani2}.
Define $\phi(\tau) \in C(M)$ by
\begin{equation} \label{7.1}
2 \sqrt{\tau} \: \phi(\tau)(m) \: = \: 
\inf_{m_0 \in M} \left( 2 \sqrt{\tau_o} \: \phi_0(m_0) \: + \:
L_-^{\tau_0, \tau}(m_0, m) \right).
\end{equation}
Define ${\mathcal E} \: : \: P(M) \rightarrow \R \cup \{ \infty \}$ 
as in (\ref{2.38}).
\begin{proposition} \label{prop16}
${\mathcal E}(c(\tau)) \: + \: \int_M \phi(\tau) \: dc(\tau) \: + \: 
\frac{n}{2} \log(\tau)$ is convex in $s = \tau^{- \: \frac12}$.
\end{proposition}
\begin{proof}
From \cite{Topping}, ${\mathcal E}(c(\tau))$ is
semiconvex in $\tau$ and its second derivative in
the Alexandrov sense satisfies
\begin{equation} \label{7.2}
\left( \tau^{\frac32} \: \frac{d}{d\tau} \right)^2 
\int_M \rho \: \log(\rho) \: \dvol_M \: \ge \:
\frac12 \: \tau^3 \: \int_M H(\nabla \phi(\tau)) \:
c(\tau) \: - \: \frac{n}{4} \: \tau.
\end{equation}
(Strictly speaking, the paper \cite{Topping} assumes that
$c(\tau_0), c(\tau_1) \in P^\infty(M)$, but the proof works
when $c(\tau_0)$ and $c(\tau_1)$ are just absolutely continuous
probability measures.)
Now
\begin{equation} \label{7.3}
\int_M \phi(\tau) \: dc(\tau) \: = \:
\int_M (\phi(\tau) \circ F_{\tau_0,\tau}) \: dc(\tau_0). 
\end{equation}
From \cite[Theorem 7.36]{Villani2}, 
for $c(\tau_0)$-almost all $m_0 \in M$ one has
\begin{equation} \label{7.4}
(\phi(\tau) \circ F_{\tau_0,\tau})(m_0) \: - \: 
(\phi(\tau_0))(m) \: = \: L_-^{\tau_0, \tau}(m_0,
F_{\tau_0,\tau}(m_0)),
\end{equation}
with $F_{\tau_0,\tau}(m_0)$ describing an ${\mathcal L}_-$-geodesic
parametrized by $\tau$. 

Given such an $m_0 \in M$,
put $\gamma(\tau) \: = \: F_{\tau_0,\tau}(m_0)$ and write
$X \: = \: \frac{d\gamma}{d\tau}$. We evaluate
$\left( \tau^{\frac32} \: \frac{d}{d\tau} \right)^2 \: \phi(\gamma(\tau))$
using formulas from \cite[Section 7]{Perelman}; see also
\cite[Section 18]{Kleiner-Lott}. 
Write $X(\tau) \: = \: \frac{d\gamma}{d\tau}$.
Then
\begin{equation} \label{7.5}
\frac{d}{d\tau} \left( 2 \: \sqrt{\tau} \: \phi(\gamma(\tau)) \right) \: = \: 
\frac{d}{d\tau} L_-^{\tau_0, \tau}(m_0, \gamma(\tau)) \: = \:
\sqrt{\tau} \left( R(\gamma(\tau), \tau) \: + \: |X(\tau)|^2 \right), 
\end{equation}
so
\begin{equation} \label{7.6}
\tau^{\frac32} \: \frac{d}{d\tau} \: \phi(\gamma(\tau)) \: = \: 
- \: \frac12 \: \sqrt{\tau} \: \phi(\gamma(\tau)) \: + \: \frac12 \:
\tau^{\frac32} \: \left( R(\gamma(\tau), \tau) \: + \: |X(\tau)|^2 \right).
\end{equation}
From \cite[(7.3)]{Perelman},
\begin{equation} \label{7.7}
\frac{d}{d\tau} \left( R(\gamma(\tau), \tau) \: + \: |X(\tau)|^2 \right)
\: = \: 
- \: H(X) \: - \: \frac{1}{\tau} \: 
\left( R(\gamma(\tau), \tau) \: + \: |X(\tau)|^2 \right).
\end{equation}
Using (\ref{7.6}) and (\ref{7.7}), one obtains
\begin{equation} \label{7.8}
\left( \tau^{\frac32} \: \frac{d}{d\tau} \right)^2 \: \phi(\gamma(\tau))
\: = \: - \: \frac12 \: \tau^3 \: H(X). 
\end{equation}

For $c(\tau_0)$-almost all $m_0 \in M$, we have
\cite[Chapter 13]{Villani2}
\begin{equation} \label{7.8.5}
X(\tau) \: = \:
\left( \nabla \phi(\tau) \right)(\gamma(\tau)).
\end{equation}
Equations (\ref{7.3}) and (\ref{7.8})
give
\begin{equation} \label{7.9}
\left( \tau^{\frac32} \: \frac{d}{d\tau} \right)^2 \int_M \phi(\tau) \: 
dc(\tau) \: = \: \int_M 
\left( H(\nabla \phi(\tau)) \circ F_{\tau_0,\tau} \right) \:
dc_0(\tau) \: = \: \int_M H(\nabla \phi(\tau)) \:
dc(\tau). 
\end{equation} 
As 
\begin{equation} \label{7.10}
\left( \tau^{\frac32} \: \frac{d}{d\tau} \right)^2 \: \log(\tau)
\: = \: \frac12 \: \tau,
\end{equation}
the proposition follows.
\end{proof}

\begin{remark}
We expect that one can prove Proposition \ref{prop16} using the
Eulerian approach and a density argument, along the lines of
\cite{Daneri-Savare (2008)}, but we do not pursue this here.
\end{remark}

We now consider the limiting case when 
$\tau_0 = 0$ and $c(0) \: = \: \delta_p$. We remark that
the preceding results of this section are valid if we just
assume that only $c(\tau_1)$ is absolutely continuous with respect to
a Riemannian volume density \cite[Chapter 13]{Villani2}.
Fix $p \in M$ and, following the notation of
\cite[Section 7]{Perelman}, put
$L(m,\tau) \: = \: L_-^{0,\tau}(p,m)$.
Choose $c(\tau_1) \in P(M)$ to be absolutely continuous
with respect to a Riemannian measure.
For each $m_1 \in M$, choose a (minimizing) ${\mathcal L}_-$-geodesic
$\gamma_{m_1} \: : \: [0, \tau_1] \rightarrow M$ with
$\gamma_{m_1}(0) = p$ and $\gamma_{m_1}(\tau_1) = m_1$. It is uniquely
defined for almost all $m_1 \in M$ \cite[Section 17]{Kleiner-Lott}. 
Let ${\mathcal R}_\tau \: : \: M \rightarrow M$
be the map given by ${\mathcal R}_\tau(m_1) \: = \: \gamma_{m_1}(\tau)$. Then
as $\tau$ ranges in $[0, \tau_1]$,
$c(\tau) \: = \: ({\mathcal R}_\tau)_* c(\tau_1)$ describes a minimizing
curve for ${\mathcal A}_-$ relative to its endpoints. If
$\tau > 0$ then $c(\tau)$ is absolutely continuous with respect to
a Riemannian volume density \cite[Chapter 13]{Villani2}.

From (\ref{7.5}),
\begin{equation} \label{7.11}
\phi(\tau) \: = \:
l(\cdot, \tau) \: = \: \frac{L(\cdot, \tau)}{2 \sqrt{\tau}}.
\end{equation}

\begin{proposition} \label{prop17}
${\mathcal E}(c(\tau)) \: + \: \int_M \phi(\tau) \: dc(\tau) \: + \: 
\frac{n}{2} \log(\tau)$ is nondecreasing in $\tau$.
\end{proposition}
\begin{proof}
Put $s = \tau^{- \: \frac12}$. If we can show that
${\mathcal E}(c(\tau)) \: + \: \int_M \phi(\tau) \: dc(\tau) \: + \: 
\frac{n}{2} \log(\tau)$ approaches a constant as $s \rightarrow \infty$,
i.e. as $\tau \rightarrow 0$, then the convexity in $s$ will imply
that 
${\mathcal E}(c(\tau)) \: + \: \int_M \phi(\tau) \: dc(\tau) \: + \: 
\frac{n}{2} \log(\tau)$ is nonincreasing in $s$, i.e. nondecreasing in $\tau$.

Let ${\mathcal L}exp(\overline{\tau}) \: : \: T_p M \rightarrow M$ be the
${\mathcal L}$-exponential map of \cite[Section 7]{Perelman}.
That is, for $V \in T_pM$, $\left( {\mathcal L}exp(\overline{\tau}) \right) (V)
\: = \: \gamma(\overline{\tau})$ where $\gamma \: : \: 
[0,\overline{\tau}] \rightarrow M$
is the ${\mathcal L}_-$-geodesic with $\gamma(0) = p$ and
$\lim_{\tau \rightarrow 0} \sqrt{\tau} \: \gamma^\prime(\tau) \: = \: V$.

Let $\Omega_{\tau_1}$ be the set of vectors $V \in T_pM$ for which
$\{{\mathcal L}exp(\tau^\prime)(V)\}_{\tau^\prime \in 
[0, \tau_1]}$ is ${\mathcal L}_-$-minimizing
relative to its endpoints. 
Put
$\widehat{c}(\tau_1) \: = \:
\left( 
{\mathcal L}exp({\tau_1})^{-1} \right)_* c(\tau_1)$, a measure
on $\Omega_{\tau_1}$.
Then $c(\tau) \: = \: {\mathcal L}exp({\tau})_* \widehat{c}(\tau_1)$.
Computing 
\begin{equation}
{\mathcal E}(c(\tau)) \: + \: \int_M \phi(\tau) \: dc(\tau) \: + \: 
\frac{n}{2} \log(\tau)
\end{equation}
with respect to the metric $g(\tau)$ on $M$ is the same as computing 
\begin{equation}
{\mathcal E}(\widehat{c}(\tau_1)) \: + \: \int_{\Omega_{\tau_1}} 
\left( \phi(\tau)
\circ {\mathcal L}exp({\tau}) \right) \: 
d\widehat{c}(\tau_1) \: + \: 
\frac{n}{2} \log(\tau)
\end{equation}
with respect to the metric
$\widehat{g}(\tau) \: = \: {\mathcal L}exp({\tau})^* g(\tau)$ on
$\Omega_{\tau_1}$.

As $\tau \rightarrow 0$, one approaches the Euclidean situation;
see \cite[Section 16]{Kleiner-Lott}.
One can check that
$\left( \phi(\tau)
\circ {\mathcal L}exp({\tau}) \right)(V)$ approaches
$|V|^2$ uniformly on the compact set $\overline{\Omega_{\tau_1}}$,
where $|V|^2$ is the norm squared of $V \in T_pM$ with respect to
$g_{T_pM}$.
Thus
\begin{equation}
\lim_{\tau \rightarrow 0} \int_{\Omega_{\tau_1}} 
\left( \phi(\tau)
\circ {\mathcal L}exp({\tau}) \right) \: 
d\widehat{c}(\tau_1) \: = \:
\int_{\Omega_{\tau_1}} 
|V|^2 \: d\widehat{c}(\tau_1).
\end{equation}
Also, $\frac{\widehat{g}(\tau)}{4\tau}$ approaches the
flat Euclidean metric $g_{T_pM}$ on $\overline{\Omega_{\tau_1}}$.
Writing $\widehat{c}(\tau_1) \: = \: \rho_1 \: \dvol(g_{T_pM})$,
for small $\tau$ 
the density of $\widehat{c}(\tau_1)$ relative to
$\dvol(\widehat{g}(\tau))$ is asymptotic to
$(4 \tau)^{- \: \frac{n}{2}} \: \rho_1$. Thus
\begin{align}
& \lim_{\tau \rightarrow 0}
\left( {\mathcal E}(\widehat{c}(\tau_1)) \: + \: 
\frac{n}{2} \log(\tau) \right) \: = \\
& \lim_{\tau \rightarrow 0} \left( \int_{\Omega_{\tau_1}}
(4 \tau)^{- \: \frac{n}{2}} \: \rho_1  \: \cdot \:
\log ((4 \tau)^{- \: \frac{n}{2}} \: \rho_1) \: \cdot
\: (4 \tau)^{\frac{n}{2}} \: \dvol_{T_pM} \: + \: 
\frac{n}{2} \log(\tau) \right) \: = \notag \\
& \int_{\Omega_{\tau_1}}
\rho_1  \: \log (\rho_1) \: \dvol_{T_pM} \: - \: 
\frac{n}{2} \log(4). \notag
\end{align}
The proposition follows.
\end{proof}

\begin{corollary} \label{cor8}
$\tau^{- \: \frac{n}{2}} \: \int_M e^{-l} \: \dvol_M$ is
nonincreasing in $\tau$. 
\end{corollary}
\begin{proof}
Given $0 < \tau^\prime < \tau^{\prime \prime} < \tau_1$,
take 
\begin{equation} \label{7.13}
c(\tau^{\prime \prime}) \: = \: \frac{e^{- \: \phi(\tau^{\prime \prime})} \: \dvol_M}{
\int_M e^{- \: \phi(\tau^{\prime \prime}}) \: \dvol_M}.
\end{equation}
Then
\begin{equation} \label{7.14}
{\mathcal E}(c(\tau^{\prime \prime})) \: + \: \int_M 
\phi(\tau^{\prime \prime}) \: dc(\tau^{\prime \prime}) \: + \: 
\frac{n}{2} \log(\tau^{\prime \prime}) \: = \:
- \log \left( (\tau^{\prime \prime})^{- \: \frac{n}{2}} \: \int_M 
e^{-\phi(\tau^{\prime \prime})} \: \dvol_M \right).
\end{equation}
Applying Proposition \ref{prop17}, with $\tau_1$ replaced by
$\tau^{\prime \prime}$, gives
\begin{equation} \label{7.15}
{\mathcal E}(c(\tau^\prime)) \: + \: \int_M \phi(\tau^\prime) \: 
dc(\tau^\prime) \: + \: 
\frac{n}{2} \log(\tau^\prime) \: \le \:
{\mathcal E}(c(\tau^{\prime \prime})) \: + \: \int_M 
\phi(\tau^{\prime \prime}) \: dc(\tau^{\prime \prime}) \: + \: 
\frac{n}{2} \log(\tau^{\prime \prime}).
\end{equation}
However,
${\mathcal E}(\mu) \: + \: \int_M \phi(\tau^\prime) \: 
d\mu \: + \: 
\frac{n}{2} \log(\tau^\prime)$
is minimized by
$- \log \left( (\tau^\prime)^{- \: \frac{n}{2}} \: 
\int_M e^{-\phi(\tau^\prime)} \: \dvol_M \right)$,
as $\mu$ ranges over probability measures that are absolutely
continuous with respect to a Riemannian measure on $M$.
Thus 
\begin{equation} \label{7.16}
- \log \left( (\tau^\prime)^{- \: \frac{n}{2}} \: 
\int_M e^{-\phi(\tau^\prime)} \: \dvol_M \right) \: \le \:
- \log \left( (\tau^{\prime \prime})^{- \: \frac{n}{2}} \: \int_M 
e^{-\phi(\tau^{\prime \prime})} \: \dvol_M \right).
\end{equation}
The corollary follows.
\end{proof}

\begin{remark} \label{rem4}
This procedure of converting a convexity statement to a
monotonicity statement works for the ${\mathcal L}_-$-cost
and the ${\mathcal L}_+$-cost but does not work for the
${\mathcal L}_0$-cost.
\end{remark}

\section{Ricci flow on a smooth metric-measure space}
\label{section8}

In this section we give a definition of Ricci flow on a smooth
metric-measure space.  Our approach is to consider the Ricci flow
on a warped product manifold $\overline{M}$ and compute the induced flow on the
base $M$. This is in analogy to what works in defining
Ricci tensors for smooth metric-measure spaces
\cite{Lott (2003)}.

It turns out that there is a $1$-parameter family of such generalized
Ricci flows, depending on a parameter $N \in [\dim(M), \infty]$.
In the case $N = \infty$, there is the curious fact that the (smooth
positive) measure
can be absorbed by diffeomorphisms of $M$, so one just reduces to
the usual Ricci flow equation on $M$.

Let $T^q$ have a fixed flat metric given in local coordinates by
$\sum_{i=1}^q dx_i^2$. 
Put $\overline{M} \: = \: M \times T^q$ with a time-dependent
warped-product metric
\begin{equation} \label{8.1}
\overline{g}(t) \: = \: \sum_{\alpha, \beta =1}^n g_{\alpha \beta}(t)
\: dx^\alpha dx^\beta \: + \: u(t)^{\frac{2}{q}} \: \sum_{i=1}^q dx_i^2.
\end{equation}
We also write $u \: = \: e^{- \: \Psi}$.
If $M$ is compact then
the pushforward of the normalized volume density
$\frac{\dvol_{\overline{M}}}{\vol(\overline{M})}$ under the projection
$\overline{M} \rightarrow M$ is $\frac{u \: \dvol_M}{\int_M u \: \dvol_M}$.

The scalar curvature $\overline{R}$ of $\overline{M}$ equals
\begin{align} \label{8.2}
R_q\: & = \: R \: - \: 2 \: u^{-1} \: \nabla^2 u \: + \: 
\left( 1 - \frac{1}{q} \right) \: u^{-2} \: |\nabla u|^2 \\
& = \: R \: + \: 2 \: \nabla^2 \Psi \: - \: 
\left( 1 + \frac{1}{q} \right) \: |\nabla \Psi|^2, \notag
\end{align}
which is the modified scalar curvature considered in \cite{Lott (2007)}.
From \cite[Section 4]{Lott2}, 
the Ricci flow equation on $\overline{M}$ becomes
\begin{align} \label{8.3}
\frac{\partial u}{\partial t} \: & = \:
\nabla^2 u, \\
\frac{\partial g_{\alpha \beta}}{\partial t} \: & = \: - \: 2 
R_{\alpha \beta} \: + \: 2 \: u^{-1} \:
u_{;\alpha \beta} \: - \: \left( 2 - \frac{2}{q} \right) \:
u^{-2} \: u_{; \alpha} \:
u_{; \beta}. \notag
\end{align}
Note that
\begin{equation}
\frac{\partial}{\partial t} (u \: \dvol_M) \: = \:
- \: R_q \: u \: \dvol_M.
\end{equation}

Equivalently,
\begin{align} \label{8.4}
\frac{\partial \Psi}{\partial t} \: & = \:
\nabla^2 \Psi \: - \: |\nabla \Psi|^2, \\
\frac{\partial g_{\alpha \beta}}{\partial t} \: & = \: - \: 2 
\left( R_{\alpha \beta} \: + \: 
\Psi_{;\alpha \beta} \: - \: \frac{1}{q} \:
\Psi_{; \alpha} \:
\Psi_{; \beta} \right). \notag
\end{align}
The right-hand side of (\ref{8.4}) involves the modified Ricci
curvature
\begin{equation} \label{8.5}
\Ric_q \: = \: \Ric \: + \: \Hess \Psi \: - \: \frac{1}{q} \:
d\Psi \otimes d\Psi
\end{equation}
considered in \cite{Lott (2003)} and \cite{Qian (1997)}.
If $u \: \dvol_M$ is a (smooth positive) probability measure then
we consider (\ref{8.4}) to be the $N$-Ricci flow equations
for the smooth metric-measure space $(M, g, u \: \dvol_M)$, with $N = n + q$.
This is in analogy to the $N$-Ricci curvature considered in
\cite{Lott-Villani}.
(If $N = n$ then we require $\Psi$ to be locally constant
and just use the usual Ricci flow equation on $M$. That is, in the
noncollapsing situation we take the measure to be the 
$n$-dimensional Hausdorff measure.)

Taking $q = \infty$, we consider the $\infty$-Ricci flow
equations to be
\begin{align} \label{8.6}
\frac{\partial \Psi}{\partial t} \: & = \:
\nabla^2 \Psi \: - \: |\nabla \Psi|^2, \\
\frac{\partial g_{\alpha \beta}}{\partial t} \: & = \: - \: 2 
\left( R_{\alpha \beta} \: + \: 
\Psi_{;\alpha \beta} \right). \notag
\end{align}
\begin{remark} \label{rem5}
The occurrence of the Bakry-\'Emery tensor on the right-hand side of
(\ref{8.6}) is different from its occurrence in Perelman's
modified Ricci flow \cite{Perelman}. In (\ref{8.6}) the function
$u = e^{- \Psi}$ satisfies a forward heat equation, whereas in Perelman's
work the corresponding measure
$e^{-f} \: \dvol_M$
satisfies a backward heat equation. 
\end{remark}

\begin{example}
We now give a trivial example of collapsing of Ricci flow solutions.
For $u \: \dvol_M \in P^\infty(M)$, put
$u_j \: = \: \frac{1}{j} \: u$. Give $\overline{M}$ the corresponding
warped-product metric $\overline{g}_j$.
Suppose that $g(t)$ and $u(t)$ satisfy (\ref{8.3}). 
Consider the corresponding
solution $(\overline{M}, \overline{g}_j(\cdot))$ to the Ricci flow equation.
For any time $t$, as $j \rightarrow \infty$,  
the metric-measure spaces
$\left( \overline{M}, \overline{g}_j(t), 
\frac{
\dvol_{\overline{M}_j}
}{
\vol(\overline{M}_j)
} \right)$ converge in
the measured Gromov-Hausdorff topology to $(M, g(t), u(t) \dvol_M)$,
which satisfies (\ref{8.3}) by construction.
\end{example}
\begin{example}
To give another example, consider the most general 
$T^q$-invariant Ricci flow on $\overline{M}$. We can write
\begin{equation} \label{8.7}
\overline{g}(t) \: = \: \sum_{\alpha, \beta =1}^n g_{\alpha \beta}(t)
\: dx^\alpha dx^\beta \: + \: \sum_{i,j=1}^q  G_{ij}(t) \: 
\left( dx^i \: + \: A^i(t) \right) \: \left( dx^j \: + \: A^i(t) \right).
\end{equation}
Put $u \: = \: \sqrt{\det(G_{ij})}$,
$X^i_{\: \: j,\alpha} \: = \: \sum_{k=1}^q 
G^{ik} \: \partial_\alpha G_{kj} \: - \: \frac{2}{q} \: u^{-1} \:
\partial_\alpha u \: \delta^i_{\: \: j}$
 and $F^i_{\alpha \beta} \: = \:
\partial_\alpha A^i_\beta \: - \: \partial_\beta A^i_\alpha$.
From \cite[Section 4]{Lott2}, 
the Ricci flow equation on $\overline{M}$ implies that the evolution
of $u$ and $g_{\alpha \beta}$ is given by
\begin{align} \label{8.8}
\frac{\partial u}{\partial t} \: & = \:
\nabla^2 u \: - \: \frac{u}{4} \: \sum g^{\alpha \gamma} \: g^{\beta \delta} \:
G_{kl} \: F^k_{\alpha \beta} \: F^l_{\gamma \delta}, \\
\frac{\partial g_{\alpha \beta}}{\partial t} \: & = \: - \: 2 
R_{\alpha \beta} \: + \: 2 \: u^{-1} \:
u_{;\alpha \beta} \: - \: \left( 2 - \frac{2}{q} \right) \:
u^{-2} \: u_{; \alpha} \:
u_{; \beta} \: + \: \sum g^{\gamma \delta} \: G_{ij} \: F^i_{\alpha \gamma} \:
F^j_{\beta \delta} \: + \: \frac12 \: \Tr(X_\alpha X_\beta). \notag
\end{align}
As before, by uniformly rescaling the torus fibers we can construct
a sequence of Ricci flow solutions
$\left( \overline{M}, \overline{g}_j(t), 
\frac{
\dvol_{\overline{M}_j}
}{
\vol(\overline{M}_j)
} \right)$ which, for each time, converge in
the measured Gromov-Hausdorff topology to $(M, g(t), u(t) \dvol_M)$,
satisfying (\ref{8.8}).
Note that instead of satisfying the $N$-Ricci flow equations (\ref{8.3}), a
solution of (\ref{8.8}) satisfies the inequalities,
\begin{align} \label{8.9}
\frac{\partial u}{\partial t} \: - \:
\nabla^2 u \: & \le \: 0, \\
\frac{\partial g_{\alpha \beta}}{\partial t} \: + \: 2 
R_{\alpha \beta} \: - \: 2 \: u^{-1} \:
u_{;\alpha \beta} \: + \: \left( 2 - \frac{2}{q} \right) \:
u^{-2} \: u_{; \alpha} \:
u_{; \beta} \: & \ge \: 0. \notag
\end{align}
End of example.
\end{example}

Returning to (\ref{8.4}),
adding a Lie derivative with respect to $\nabla \Psi$ to the
right-hand side
gives the equations
\begin{align} \label{8.10}
\frac{\partial \Psi}{\partial t} \: & = \:
\nabla^2 \Psi, \\
\frac{\partial g_{\alpha \beta}}{\partial t} \: & = \: - \: 2 
R_{\alpha \beta} \: + \: \frac{2}{q} \:
\Psi_{; \alpha} \:
\Psi_{; \beta}. \notag
\end{align}
Note that
\begin{equation}
\frac{\partial}{\partial t} \left( e^{- \Psi} \: \dvol_M \right) \: = \: 
- \: \Tr(\Ric_q) \: e^{- \Psi} \: \dvol_M.
\end{equation}

In particular, the $\infty$-Ricci flow equations (\ref{8.6})
become
\begin{align} \label{8.11}
\frac{\partial \Psi}{\partial t} \: & = \:
\nabla^2 \Psi, \\
\frac{\partial g_{\alpha \beta}}{\partial t} \: & = \: - \: 2 
R_{\alpha \beta}. \notag
\end{align}
That is, we obtain
a forward heat equation coupled to an ordinary Ricci flow.

We now consider convexity of the entropy function for the system
(\ref{8.3}), where the entropy is computed relative to the
background measure $u \: \dvol_M$.
Consider the transport equations on $M$ :
\begin{align} \label{8.12}
\frac{\partial \rho}{\partial t} \: & = \: - \: u^{-1} \: \nabla^\alpha
(\rho u \nabla_\alpha \phi ) \: + \: \overline{R} \: \rho, \\
\frac{\partial \phi}{\partial t} \: & = \: - \: \frac12 \:
|\nabla \phi|^2 \: + \: \frac12 \: \overline{R}. \notag
\end{align}
Note that $\int_M \rho \: u \: \dvol_M$ is constant in $t$, so
we can take $\rho \: u \: \dvol_M$ to be a probability measure.
Applying
Corollary \ref{cor4} to $\overline{M}$ implies that 
if (\ref{8.3}) and (\ref{8.12}) are satisfied 
then $\int_M \left( \rho \: \log(\rho) \: - \: \phi \: \rho \right) 
\: u \: \dvol_M$ is convex in $t$.
Equivalently, in terms of the equations (\ref{8.10}),
if $\rho$ and $\phi$ satisfy
\begin{align} \label{8.13}
\frac{\partial \rho}{\partial t} \: & = \: - \: \nabla^\alpha
(\rho \nabla_\alpha \phi ) \: + \: \langle \nabla \Psi, \nabla \rho \rangle
\: + \: \langle \nabla \Psi, \nabla \phi \rangle \: \rho
\: + \: \overline{R} \: \rho, \\
\frac{\partial \phi}{\partial t} \: & = \: - \: \frac12 \:
|\nabla \phi|^2 \: + \: \langle \nabla \Psi, \nabla \phi \rangle 
\: + \: \frac12 \: \overline{R} \notag
\end{align}
then $\int_M \left( \rho \: \log(\rho) \: - \: \phi \: \rho \right) 
\: e^{-\Psi} \: \dvol_M$ is convex in $t$.

When $q \rightarrow \infty$, so that (\ref{8.11}) holds, we claim that
this convexity is no more than the convexity of Corollary \ref{cor4}
when applied to $M$,
after a change of variables.
Namely, when $q \rightarrow \infty$, if we put
\begin{align} \label{8.14}
\widetilde{\rho} \: & = \: e^{- \Psi} \: \rho, \\
\widetilde{\phi} \: & = \: \phi \: - \: \Psi \notag
\end{align}
then the equations (\ref{8.13}) are equivalent to
\begin{align} \label{8.15}
\frac{\partial \widetilde{\rho}}{\partial t} \: & = \: - \: \nabla^\alpha
(\widetilde{\rho} \nabla_\alpha \widetilde{\phi} ) \: + \:
R \: \widetilde{\rho}, \\
\frac{\partial \widetilde{\phi}}{\partial t} \: & = \: - \: \frac12 \:
|\nabla \widetilde{\phi}|^2
\: + \: \frac12 \: R. \notag
\end{align}
From Corollary \ref{cor4}, we know that
$\int_M \left( \widetilde{\rho} \: \log(\widetilde{\rho}) \: - \: 
\widetilde{\phi} \: \widetilde{\rho} \right) \: \dvol_M$ is convex in
$t$.
This is the same as saying that
$\int_M \left( \rho \: \log(\rho) \: - \: \phi \: \rho \right) 
\: e^{-\Psi} \: \dvol_M$ is convex in $t$.

To summarize, for each $N \in [n, \infty]$ there is a $N$-Ricci
flow (\ref{8.3}). Its right-hand side involves the $N$-Ricci curvature
tensor.  A background solution of the $N$-Ricci flow equation implies a 
convexity result for the
transport equations (\ref{8.12}). In the special case when
$N = \infty$, one can decouple the (smooth positive) measure 
within the metric flow by performing
diffeomorphisms, to recover a forward heat equation coupled
to the usual Ricci flow (\ref{8.6}).

\appendix
\section{The ${\mathcal L}_+$-entropy}
\label{section9}

In this section we give the analogs of Sections 
\ref{section6} and \ref{section7}
for the ${\mathcal L}_+$-functional that was considered in
\cite{Feldman-Ilmanen-Ni (2005)}. This is for possible future
reference. We reprove the
monotonicity of the Ilmanen-Feldman-Ni forward reduced volume.

Let $M$ be a connected closed manifold and let $g(\cdot)$ be a Ricci flow
solution on $M$, i.e. (\ref{3.1}) is satisfied.

\begin{definition}
If $\gamma \: : \: [t^\prime, t^{\prime \prime}] \rightarrow M$ is a 
smooth curve with $t^\prime > 0$ then
its ${\mathcal L}_+$-length is
\begin{equation} \label{A.1}
{\mathcal L}_+(\gamma) \: = \: \frac12 \int_{t^\prime}^{t^{\prime \prime}}
\sqrt{t} \:
\left( g \left( \frac{d\gamma}{dt}, \frac{d\gamma}{dt} \right) \: + \:
R(\gamma(t),t) \right) \: dt, 
\end{equation}
where the time-$t$ metric $g(t)$ is used to define the integrand.

Let $L_+^{t^\prime, t^{\prime \prime}}(m^\prime, m^{\prime \prime})$
be the infimum of ${\mathcal L}_+$ over curves $\gamma$ with
$\gamma(t^\prime) = m^\prime$ and
$\gamma(t^{\prime \prime}) = m^{\prime \prime}$.
\end{definition}

The Euler-Lagrange equation for the ${\mathcal L}_+$-functional
is easily derived to be
\begin{equation} \label{A.2}
\nabla_{\frac{d\gamma}{dt}} \left( \frac{d\gamma}{dt} \right) 
\: - \: \frac12 \: \nabla R
\: + \: \frac{1}{2t} \: \frac{d\gamma}{dt} 
\: - \: 2 \: \Ric \left( \frac{d\gamma}{dt},\cdot \right) \: = \: 0.
\end{equation}
The ${\mathcal L}_+$-exponential map is defined by saying that
for $V \in T_{m^\prime}M$, one has
\begin{equation} \label{A.3}
{\mathcal L}_+\exp_{m^\prime}^{t^\prime, t^{\prime \prime}} 
\left(V \right) \: = \: \gamma(t^{\prime \prime})
\end{equation}
where $\gamma$ is the solution to (\ref{A.2}) with
$\gamma(t^\prime) = m^\prime$ and $\frac{d\gamma}{dt} \Big|_{t = t^\prime}
= V$.

\begin{definition}
Given $\mu^\prime, \mu^{\prime \prime} \in P(M)$, put
\begin{equation} \label{A.4}
C_+^{t^\prime, t^{\prime \prime}}
(\mu^\prime, \mu^{\prime \prime}) \: = \: \inf_\Pi
\int_{M \times M} 
L_+^{t^\prime, t^{\prime \prime}}(m^\prime, m^{\prime \prime}) \:
d\Pi(m^\prime, m^{\prime \prime}),
\end{equation}
where $\Pi$ ranges over the elements of $P(M \times M)$ whose pushforward
to $M$ under projection onto the first (resp. second) factor is
$\mu^\prime$ (resp. $\mu^{\prime \prime}$).
Given a continuous curve $c \: : \: [t^\prime, t^{\prime \prime}] \rightarrow
P(M)$, put
\begin{equation} \label{A.5}
{\mathcal A}_+(c) 
\: = \: \sup_{J \in \Z^+} \sup_{t^\prime = t_0 \le t_1 \le \ldots
\le t_J = t^{\prime \prime}} \sum_{j=1}^J C_+^{t_{j-1}, t_j}(c(t_{j-1}),
c(t_j)).
\end{equation}
\end{definition}

We can think of 
${\mathcal A}_+$ as a generalized length functional associated
to the generalized metric $C_+$. 
By \cite[Theorem 7.21]{Villani2}, 
${\mathcal A}_+$ is a coercive action on $P(M)$
in the sense of
\cite[Definition 7.13]{Villani2}. In particular,
\begin{equation} \label{A.6}
C_+^{t^\prime, t^{\prime \prime}}
(\mu^\prime, \mu^{\prime \prime}) \: = \:
\inf_c {\mathcal A}_+(c),
\end{equation}
where $c$ ranges over continuous curves with $c(t^\prime) = \mu^\prime$ and
$c(t^{\prime \prime}) = \mu^{\prime \prime}$.

If $c \: : [t_0, t_1] 
\rightarrow P^\infty(M)$ is a smooth curve in $P^\infty(M)$ with
$t_0 > 0$ then 
we write $c(t) \: = \: \rho(t) \: \dvol_M$ and let $\phi(t)$ satisfy
\begin{equation} \label{A.7}
\frac{\partial \rho}{dt} \: = \: - \: 
\nabla^i \left( \rho \nabla_i \phi \right) \: + \: R \rho.
\end{equation}
Note that $\phi(t)$ is uniquely defined up to an additive constant.
The scalar curvature term in (\ref{A.7})  ensures that
\begin{equation} \label{A.8}
\frac{d}{dt} \int_M \rho \: \dvol_M \: = \: 0.
\end{equation}

Consider the Lagrangian 
\begin{equation} \label{A.9}
E_+(c) \: = \:
\int_{t_0}^{t_1} \int_M \sqrt{t} \left( |\nabla \phi|^2
\: + \: R \right) \: \rho \: \dvol_M \: dt,
\end{equation}
where the integrand at time $t$ is computed using $g(t)$.

\begin{proposition} \label{prop18}
Let 
\begin{equation} \label{A.10}
\rho \: \dvol_M \: : \: [t_0, t_1] 
\times [- \epsilon, \epsilon]
\rightarrow P^\infty(M)
\end{equation}
be a smooth map, with $\rho \equiv \rho(t,u)$.
Let 
\begin{equation} \label{A.11}
\phi \: : \: [t_0, t_1] \times [- \epsilon, \epsilon]
\rightarrow C^\infty(M)
\end{equation}
be a smooth map that satisfies (\ref{A.7}),
with $\phi = \phi(t,u)$.
Then 
\begin{align} \label{A.12}
\frac{dE_+}{du} \Bigg|_{u = 0} \: = \:
& 2 \sqrt{t} \int_M \phi \: \frac{\partial \rho}{\partial u}
\: \dvol_M \Bigg|_{t=t_0}^{t_1} \: - \\
& 2 \: 
\int_{t_0}^{t_1} 
\int_M \sqrt{t} \:
\left( \frac{\partial \phi}{\partial t} \: + \: \frac12 \:
|\nabla \phi|^2  \: - \: \frac12 \: 
R \: + \: \frac{1}{2t} \: \phi \right) \: \frac{\partial \rho}{\partial u}
\: \dvol_M \: dt, \notag
\end{align}
where the right-hand side is evaluated at $u = 0$.
\end{proposition}
\begin{proof}
The proof is similar to that of Proposition \ref{prop14}. We omit the details.
\end{proof}

From (\ref{A.12}), the Euler-Lagrange equation for $E_+$ is 
\begin{equation} \label{A.13}
\frac{\partial \phi}{\partial t} \: = \: - \: \frac12 \: |\nabla \phi|^2
\: + \: \frac12 \: R \: - \: \frac{1}{2t} \: \phi \: + \: \alpha(t),
\end{equation}
where $\alpha \in C^\infty([t_0, t_1])$. Changing $\phi$ by a
spatially-constant function, we can assume that $\alpha = 0$, so
\begin{equation} \label{A.14}
\frac{\partial \phi}{\partial t} \: = \: - \: \frac12 \: |\nabla \phi|^2
\: + \: \frac12 \: R \: - \: \frac{1}{2t} \: \phi.
\end{equation}
If a smooth curve in $P^\infty(M)$ minimizes 
$E_+$, relative to its endpoints, then it will satisfy (\ref{A.14}).
Given $t_0  \le t^\prime < t^{\prime \prime} \le t_1$,
the viscosity solution of 
(\ref{A.14}) satisfies
\begin{equation} \label{A.15}
2 \sqrt{t^{\prime \prime}} \:
\phi(t^{\prime \prime})(m^{\prime \prime}) \: = \: 
\inf_{m^\prime \in M} \left( 
2 \sqrt{t^\prime} \: \phi(t^\prime)(m^\prime) \: + \:
L_+^{t^\prime, t^{\prime \prime}}(m^\prime, m^{\prime \prime})
\right). 
\end{equation}
Then the solution of (\ref{A.7}) satisfies 
\begin{equation} \label{A.16}
\rho(t^{\prime \prime}) \: \dvol_M \: = \: 
(F_{t^\prime, t^{\prime \prime}})_* (\rho(t^\prime) \: \dvol_M), 
\end{equation}
where the transport map
$F_{t^\prime, t^{\prime \prime}} \: : \: M \rightarrow M$ is given by
\begin{equation} \label{A.17}
F_{t^\prime, t^{\prime \prime}}(m^\prime) \: = \: 
{\mathcal L}_+\exp_{m^\prime}^{t^\prime, t^{\prime \prime}} 
\left( \nabla_{m^\prime}
\phi(t^\prime) \right).
\end{equation}

\begin{proposition} \label{prop19}
Suppose that (\ref{A.7}) and (\ref{A.14}) are satisfied.
Then
\begin{equation} \label{A.18}
\frac{d}{dt} \int_M \phi \rho \: \dvol_M \: = \:
\frac12 \int_M \left( |\nabla \phi|^2 \: + \: R \right) \: \rho \: \dvol_M
\: - \: \frac{1}{2t} \int_M \phi \rho \: \dvol_M,
\end{equation}

\begin{align} \label{A.19}
\frac12 \: \frac{d}{dt} \int_M |\nabla \phi|^2 \: 
\rho \: \dvol_M \: = \: & \int_M \left( \Ric(\nabla \phi, \nabla \phi)
\: + \: \frac12 \: \langle \nabla R, \nabla \phi \rangle \right) \: \rho \: 
\dvol_M \: - \\
& \frac{1}{2t} \: \int_M |\nabla \phi|^2 \: 
\rho \: \dvol_M, \notag
\end{align}

\begin{equation} \label{A.20}
\frac{d}{dt} \int_M \rho \: \log(\rho) \: 
\dvol_M \: = \: \int_M \left( \langle \nabla \rho, \nabla \phi \rangle \: + \: 
R \: \rho \right) 
\dvol_M,
\end{equation}

\begin{equation} \label{A.21}
\frac{d}{dt} \int_M R \rho \: 
\dvol_M \: = \: \int_M \left( R_t \: + \: 
\langle \nabla R, \nabla \phi \rangle \right) \: \rho \: 
\dvol_M
\end{equation}

and
\begin{align} \label{A.22}
& \frac{d}{dt} \int_M \langle \nabla \rho, \nabla \phi \rangle \: 
\dvol_M \: = \\
& \int_M \left( |\Hess \phi|^2 \: + \: \Ric(\nabla \phi, \nabla \phi)
\: - \: 2 \: \langle \Ric, \Hess \phi \rangle \: - \: \frac12 
\: \nabla^2 R \right) \rho \: \dvol_M \: - \notag \\
& \frac{1}{2t} \: \int_M \langle \nabla \rho, \nabla \phi \rangle \: 
\dvol_M. \notag
\end{align}
\end{proposition}
\begin{proof}
The proof is similar to that of Proposition \ref{prop15}. We omit the details.
\end{proof}

\begin{corollary} \label{cor9}
Under the hypotheses of Proposition \ref{prop19},
\begin{equation} \label{A.23}
\left( t^{\frac12} \: \frac{d}{dt} \right)^2 
\int_M \rho \: \log(\rho) \: \dvol_M \: = \:
t \int_M \left( |\Ric \: - \: \Hess \phi|^2 \: + \:
\frac12 \: H(\nabla \phi) \right) \:
\rho \: \dvol_M,
\end{equation}
where
\begin{equation} \label{A.24}
H(X) \: = \: R_t \: + \: 2 \langle \nabla R, X \rangle \: + \:
2 \: \Ric(X,X) \: + \: \frac{R}{t}
\end{equation}
is Hamilton's trace Harnack expression.
Also,
\begin{align} \label{A.25}
& \left( t^{\frac32} \: \frac{d}{dt} \right)^2 
\left(  \int_M (\rho \: \log(\rho) \: - \: 
\phi \: \rho ) \: \dvol_M \: + \: \frac{n}{2} \: \log(t) \right) \: = 
\\
&
t^3 \: \int_M \left| \Ric \: - \: \Hess \phi \: + \: \frac{g}{2t}
\right|^2 \:
\rho \: \dvol_M. \notag
\end{align}
In particular, $\int_M \left( \rho \: \log(\rho) \: - \: 
\phi \: \rho \right) \: \dvol_M \: + \: 
\frac{n}{2} \: \log(t)$ is convex in $t^{- \: \frac12}$.
\end{corollary}
\begin{proof}
This follows from Proposition \ref{prop19}, along with the
curvature evolution equation (\ref{4.34}).
\end{proof}

Let $c \: : \: [t_0, t_1] \rightarrow P(M)$ be a minimizing
curve for ${\mathcal A}_+$ relative to its endpoints. We
assume that $c(t_0)$ are $c(t_1)$ are 
absolutely continuous with respect to 
a Riemannian volume density on $M$.
Then $c(t) \: = \: (F_{t_0, t})_* c(t_0)$, where there is a
semiconvex function $\phi_0 \in C(M)$ so that
$F_{t_0, t}(m_0) \: = \: {\mathcal L}_+ \exp_{m_0}^{t_0, t}
(\nabla_{m_0} \phi_0)$ \cite{Bernard-Buffoni (2007)},
\cite[Chapters 10,13]{Villani2}.
Define $\phi(t) \in C(M)$ by
\begin{equation} \label{A.27}
2 \sqrt{t} \: \phi(t)(m) \: = \: 
\inf_{m_0 \in M} \left( 2 \sqrt{t_o} \: \phi_0(m_0) \: + \:
L_+^{t_0, t}(m_0, m) \right).
\end{equation}
Define ${\mathcal E} \: : \: P(M) \rightarrow \R \cup \{ \infty \}$ 
as in (\ref{2.38}).
\begin{proposition} \label{prop20}
${\mathcal E}(c(t)) \: - \: \int_M \phi(t) \: dc(t) \: + \: 
\frac{n}{2} \log(t)$ is convex in $s = t^{- \: \frac12}$.
\end{proposition}
\begin{proof}
The proof is similar to that of Proposition \ref{prop16}.  We omit
the details.
\end{proof}

We now consider the limiting case when 
$t_0 = 0$ and $c(0) \: = \: \delta_p$.
Fix $p \in M$.
Choose $c(t_1) \in P(M)$ to be absolutely continuous
with respect to a Riemannian measure.
For each $m_1 \in M$, choose a (minimizing) ${\mathcal L}_+$-geodesic
$\gamma_{m_1} \: : \: [0, t_1] \rightarrow M$ with
$\gamma_{m_1}(0) = p$ and $\gamma_{m_1}(t_1) = m_1$. It is uniquely
defined for almost all $m_1 \in M$. 
Let ${\mathcal R}_t \: : \: M \rightarrow M$
be the map given by ${\mathcal R}_t(m_1) \: = \: \gamma_{m_1}(t)$. Then
as $t$ ranges in $[0, t_1]$,
$c(t) \: = \: ({\mathcal R}_t)_* c(t_1)$ describes a minimizing
curve for ${\mathcal A}_+$ relative to its endpoints.

Take
\begin{equation} \label{A.28}
\phi(t) \: = \:
l_+(\cdot, t) \: = \: \frac{L_+^{0,t}(p,\cdot)}{2 \sqrt{t}}.
\end{equation}

\begin{proposition} \label{prop21}
${\mathcal E}(c(t)) \: - \: \int_M \phi(t) \: dc(t) \: + \: 
\frac{n}{2} \log(t)$ is nondecreasing in $t$.
\end{proposition}
\begin{proof}
The proof is similar to that of Proposition \ref{prop17}. We omit the details.
\end{proof}

\begin{corollary} \label{cor10}
$t^{- \: \frac{n}{2}} \: \int_M e^{l_+} \: \dvol_M$ is
nonincreasing in $t$. 
\end{corollary}
\begin{proof}
The proof is similar to that of Corollary \ref{cor8}. We omit the details.
\end{proof}

\begin{remark} \label{rem6}
In the Euclidean case, $l_+(x, t) \: = \: \frac{|x|^2}{4t}$.
Because $l_+$ occurs with a positive
sign in the exponential in Corollary \ref{cor10}, 
we cannot expect $t^{- \: \frac{n}{2}} \: \int_M e^{l_+} \: \dvol_M$
to make sense if $M$ is noncompact. This is in contrast to
what happens for Perelman's reduced volume 
$\tau^{- \: \frac{n}{2}} \: \int_M e^{- \: l} \: \dvol_M$, which makes
sense if the Ricci flow has bounded sectional curvature on 
compact time intervals.
\end{remark}

\end{document}